\documentclass[11pt]{amsart}
\usepackage{amsfonts} 
\usepackage{amssymb}
\usepackage{amsthm}
\usepackage{amsmath} 
\usepackage{mathrsfs}
\usepackage{mathtools}
\usepackage{dsfont}
\usepackage{esint}
\usepackage{ulem}
\usepackage{marginnote}
\usepackage{array}
\usepackage{algorithm}
\usepackage{algpseudocode}
\usepackage{soul}
\usepackage{comment}

\usepackage{booktabs}

\usepackage{longtable}

\usepackage[all]{xy}

\usepackage[english]{babel} 
\usepackage[left=2cm,right=2cm,top=3.5cm,bottom=3cm,headsep=0.7cm]{geometry}
\usepackage{xargs}

\usepackage{csquotes}
\usepackage{stmaryrd}
\usepackage[colorinlistoftodos,prependcaption,textsize=tiny]{todonotes}
\newcommandx{\unsure}[2][1=]{\todo[linecolor=red,backgroundcolor=red!25,bordercolor=red,#1]{#2}}
\newcommandx{\change}[2][1=]{\todo[linecolor=blue,backgroundcolor=blue!25,bordercolor=blue,#1]{#2}}
\newcommandx{\info}[2][1=]{\todo[linecolor=OliveGreen,backgroundcolor=OliveGreen!25,bordercolor=OliveGreen,#1]{#2}}
\newcommandx{\improvement}[2][1=]{\todo[linecolor=Plum,backgroundcolor=Plum!25,bordercolor=Plum,#1]{#2}}
\newcommandx{\thiswillnotshow}[2][1=]{\todo[disable,#1]{#2}}

\usepackage{pgf,tikz}
\usetikzlibrary{positioning, calc, arrows.meta}
\usetikzlibrary{arrows}
\usepackage{subcaption}
\usepackage{graphicx}
\usepackage[parfill]{parskip}

\usepackage{enumerate}
\usepackage{enumitem}

\usepackage{hyperref}
\usepackage[nameinlink]{cleveref}
\usepackage{mathtools}
\hypersetup{
    colorlinks,
    linkcolor={red!80!black},
    citecolor={blue!80!black}
}

\crefname{assumption}{assumption}{assumptions}  
\Crefname{assumption}{Assumption}{Assumptions}  
\crefname{counterexample}{counterexample}{counterexamples}
\Crefname{counterexample}{Counterexample}{Counterexamples}

\theoremstyle{plain}
\begingroup
\theoremstyle{plain}
\newtheorem{theorem}{Theorem}[section]
\newtheorem{corollary}[theorem]{Corollary}
\newtheorem{proposition}[theorem]{Proposition}
\newtheorem{lemma}[theorem]{Lemma}
\theoremstyle{definition}
\newtheorem{definition}[theorem]{Definition}

\theoremstyle{remark}
\newtheorem{remark}[theorem]{Remark}

\endgroup

\theoremstyle{definition}
\theoremstyle{remark}

\numberwithin{equation}{section}

\newcommand{\N}{\mathbb{N}}

\mathchardef\emptyset="001F

\newcommand{\SCIG}{\operatorname{SCI}_G}

\newcommand{\Ev}{\operatorname{Ev}}

\title[\bf Family Exactness For The SCI]{\bf From Witness-Space Sharpness To Family-Pointwise Exactness For The Solvability Complexity Index}

\begin{document}

\author[C.~Sorg]{Christopher Sorg$^1$}
	\address[C.~Sorg]{
		\textup{Chair for Theoretical computer science, mathematics, and operations research}
		\newline \indent
		\textup{Department of Computer Science} \newline \indent
		\textup{University of the Bundeswehr Munich}
		\newline \indent
		\textup{85577 Neubiberg, Germany}}

\email{{\href{mailto:chr.sorg@unibw.de}{\textcolor{blue}{\texttt{chr.sorg@unibw.de}}}}
}

\footnotetext[1]{Inf1, University of the Bundeswehr Munich, Werner-Heisenberg-Weg 39, 85577 Neubiberg, Germany}

\begin{abstract}
We study how exact Solvability Complexity Index (SCI) statements should be formulated for
families of computational problems rather than for single problems. While the equality
\(\SCIG (\mathcal P)=k\) is unambiguous for an individual computational problem \(\mathcal P\), the family setting
requires one to distinguish family-pointwise exactness, witness-space sharpness, and worst-case
exactness. We formalize this trichotomy, prove that witness-space sharpness coincides with
worst-case exactness but is, in general, strictly weaker than family-pointwise exactness, and give a canonical
source-family example witnessing the strictness. We then establish two positive upgrade theorems: an abstract pullback principle and a
concrete finite-query criterion guaranteeing that witness-space sharpness upgrades to
family-pointwise exactness. Next, we introduce a decoder-regular finite-query transport preorder
on SCI computational problems, prove that it is a preorder, derive a transport-saturation
sufficient criterion extending the principal-source package, and show that the associated
transport degrees need not form a lattice in full generality. We analyze the natural decoder
classes \(\mathscr R_{\mathrm{cont}}\) and \(\mathscr R_{\mathrm{Bor}}\): on the full class the
corresponding quotients are not upper semilattices, while on the nondegenerate subclass the
preorder is upward and downward directed. Finally, we exhibit two natural positive families
realizing the principal transport mechanism: exact integration on compact intervals and a
fixed-window spectral decision family obtained by block-diagonal stabilization.
\end{abstract}

\maketitle

\begin{center}\small
\end{center}
\tableofcontents

\section{Introduction}

For a computational problem \(\mathcal P\), let \(\SCIG(\mathcal P)\) denote its type-\(G\)
solvability complexity index in the sense of \cite[Def.~2.6, Def.~2.10 and Def.~2.19]{Sor26}. If
\(\mathcal F=(\mathcal P_i)_{i\in I}\) is a nonempty family and \(k\in\mathbb N_0\), then three distinct
family-level sharpness notions arise: family-pointwise exactness means
\(\SCIG(\mathcal P_i)=k\) for every \(i\in I\); witness-space sharpness means
\(\SCIG(\mathcal P_i)\le k\) for every \(i\in I\) and
\(\SCIG(\mathcal P_{i^\ast})=k\) for some \(i^\ast\in I\); worst-case exactness
means \(\sup_{i\in I}\SCIG(\mathcal P_i)=k\). The spectral origins of the SCI
hierarchy \cite{hansen2011solvability, BACH15, ColbrookHansen22} show that these
notions need not coincide. The first aim of this paper is therefore to formulate the
following upgrade problem precisely: under which natural hypotheses does witness-space
sharpness imply family-pointwise exactness?

A second aim is structural. The exact-input SCI framework allows one to formulate family-level
lower-bound transfer in a way that separates three ingredients: a hard source problem, a
transport mechanism from that source into all family members, and a family-wide upper bound.
This leads naturally to finite-query evaluation reductions, to a decoder-regular finite-query
transport preorder on SCI computational problems, and to a transport-saturation sufficient
criterion for family-pointwise exactness.

The main results are as follows. \Cref{sec:SCIdefs} recalls the raw type-\(G\) SCI framework and
introduces the three family-level exactness notions. \Cref{sec:strictness-source-family} gives an elementary
strictness example showing that witness-space sharpness does not imply family-pointwise
exactness in arbitrary families. \Cref{sec:AbstrUpgrTh} proves an abstract pullback principle and a
concrete finite-query criterion.  It then introduces the decoder-regular finite-query
transport preorder, proves that it is a preorder, derives a transport-saturation
criterion extending the principal-source package, and shows that the corresponding
transport-degree structures need not be lattices in full generality. \Cref{sec:interinteupgrade} supplies
two natural families where the positive criterion applies: exact integration on compact
intervals and a fixed-window spectral decision family obtained by block-diagonal
stabilization. \Cref{sec:openproblems} records limitations and open problems. \Cref{sec:LattTranspQuot} analyzes
the transport-degree quotients for the natural decoder classes \(\mathscr R_{\mathrm{cont}}\) and
\(\mathscr R_{\mathrm{Bor}}\).

This article is written as a companion to \cite{Sor26}.  The companion paper develops
the raw and implemented SCI framework, the finite-query collapse mechanism, and
canonical finite-height source problems.  The present paper is largely self-contained at
the raw type-\(G\) level, but it uses the terminology and notation of~\cite{Sor26} for
SCI computational problems, general algorithms, standard raw type-\(G\) towers, and
finite-query factorisation.

\textbf{Notation.} Throughout, \(\mathbb N:=\{1,2,3,\dots\}\) and \(\mathbb N_0:=\{0,1,2,\dots\}\).

\section{SCI Preliminaries And Family-Level Exactness}{\label{sec:SCIdefs}}
\begin{definition}[Computational problem {\cite[Def.~2.4]{Sor26}}]\label{def:SCIcompProb}
A computational problem is a tuple
\[
\mathcal P=(\Xi,\Omega,(\mathcal M,d),\Lambda),
\]
where \(\Omega\) is the input class, \((\mathcal M,d)\) is a metric output space, \(\Xi:\Omega\to \mathcal M\) is the target map, and \(\Lambda\) is a family of complex-valued evaluation maps on \(\Omega\), subject to the consistency assumption
\[
\Xi(A)\neq \Xi(B)\ \Longrightarrow\ \exists f\in\Lambda\text{ with }f(A)\neq f(B).
\]
Equivalently, \(\Xi\) factors through the full evaluation table \(\Ev_\Lambda\) by \cite[Thm.~3.2]{Sor26}.
\end{definition}

\begin{definition}[General algorithm, standard raw type-\(G\) tower, and \(\SCIG\) {\cite[Def.~2.6, Def.~2.10 and Def.~2.19]{Sor26}}]\label{def:GenAlgoSCI}
Let
\[
        \mathcal P=(\Xi,\Omega,(\mathcal M,d),\Lambda)
\]
be a computational problem.

\begin{enumerate}
\item[(i)] A \textit{general algorithm} on \(\mathcal P\) is a pair
\[
        (\Gamma,\Lambda_\Gamma), \quad \Gamma:\Omega\to \mathcal M,\quad \Lambda_\Gamma:\Omega\to[\Lambda]^{<\omega}\setminus\{\varnothing\},
\]
such that for all \(A,B\in\Omega\),
\[
        \bigl(\forall f\in\Lambda_\Gamma(A)\; f(B)=f(A)\bigr) \Longrightarrow
        \bigl(\Gamma(B)=\Gamma(A)\ \text{and}\ \Lambda_\Gamma(B)=\Lambda_\Gamma(A)\bigr).
\]

\item[(ii)] A \textit{standard raw type-\(G\) tower of height \(0\)} for \(\mathcal P\) is a general algorithm \((\Gamma,\Lambda_\Gamma)\) with
\[
        \Gamma(A)=\Xi(A)
\]
for $A\in\Omega$.

\item[(iii)] For \(k\in \mathbb{N}\), a \textit{standard raw type-\(G\) tower of height \(k\)} for \(\mathcal P\) is a family of general algorithms
\[
        \Gamma_{n_k,\ldots,n_1}:\Omega\to \mathcal M
\]
for $(n_1,\ldots,n_k)\in\mathbb N^k$, such that for every \(A\in\Omega\),
\[
        \Xi(A)= \lim_{n_k\to\infty} \lim_{n_{k-1}\to\infty} \cdots \lim_{n_1\to\infty} \Gamma_{n_k,\ldots,n_1}(A)
\]
in the metric \(d\). The intermediate maps
\[
        \Gamma_{n_k},\Gamma_{n_k,n_{k-1}},\ldots,\Gamma_{n_k,\ldots,n_2}
\]
are only semantic pointwise limits. They are \textit{not} required to be general algorithms.

\item[(iv)] The raw type-\(G\) solvability complexity index is
\[
        \SCIG(\mathcal P) := \min\{k\in\mathbb N_0: \mathcal P\text{ has a standard raw type-\(G\) tower of height }k\},
\]
with value \(\infty\) if no such \(k\) exists.
\end{enumerate}
\end{definition}

\begin{remark}[Convention on suppressed query maps]
When a general algorithm is denoted only by its output map $\Gamma$, the accompanying finite-query map $\Lambda_\Gamma$ is understood implicitly.
Likewise, in a type-$G$ tower we often display only the output maps and suppress the associated query-set maps.
\end{remark}

\begin{definition}[Finite-query factorisation {\cite[Def.~7.3]{Sor26}}]
Let \(\mathcal P=(\Xi,\Omega,(\mathcal M,d),\Lambda)\). We say that \(\Xi\) admits a \textit{finite-query factorisation through \(\Lambda\)} if there exist \(f_1,\dots,f_m\in \Lambda\) and a map
\[
G:\operatorname{im}(f_1,\dots,f_m)\to \mathcal M
\]
such that
\[
\Xi(A)=G(f_1(A),\dots,f_m(A))\qquad(A\in\Omega).
\]
\end{definition}

\begin{theorem}[Finite-query collapse {\cite[Thm.~7.4]{Sor26}}]{\label{thm:FiniteQueryColl}}
If \(\Xi\) admits a finite-query factorisation through \(\Lambda\), then
\[
\SCIG(\mathcal P)=0.
\]
\end{theorem}

\begin{definition}[Family-level upper and lower bounds]
Let \(\mathcal F=(\mathcal P_i)_{i\in I}\) be a nonempty family of computational problems.
For \(k\in \N_0\cup\{\infty\}\), write
\[
\mathrm{UB}_k(\mathcal F)\;:\Longleftrightarrow\; \forall i\in I,\ \SCIG(\mathcal P_i)\leq k,
\]
\[
\mathrm{LB}_k(\mathcal F)\;:\Longleftrightarrow\; \forall i\in I,\ \SCIG(\mathcal P_i)\geq k.
\]
\end{definition}

Throughout the paper, every family \(\mathcal F=(\mathcal P_i)_{i\in I}\) is assumed to be nonempty.
\begin{definition}[The three sharpness notions]{\label{def:sharpNot}}
Let \(\mathcal F=(\mathcal P_i)_{i\in I}\) and \(k\in \N_0\).
\begin{enumerate}[label=(\roman*)]
	\item \textbf{Exactness for a fixed problem:}
	For a single computational problem \(\mathcal P\), \textit{exact type-$G$ height \(k\)} means
	\[
	\SCIG(\mathcal P)=k.
	\]
	
	\item \textbf{Family-pointwise exactness at \(k\):}
	We say that \(\mathcal F\) is pointwise exact at \(k\) if
	\[
	\forall i\in I,\ \SCIG(\mathcal P_i)=k.
	\]
	Equivalently, \(\mathcal F\) satisfies both \(\mathrm{UB}_k(\mathcal F)\) and \(\mathrm{LB}_k(\mathcal F)\).
	
	\item \textbf{Witness-space sharpness at \(k\):}
	We say that \(\mathcal F\) is witness-space sharp at \(k\) if
	\[
	\forall i\in I,\ \SCIG(\mathcal P_i)\leq k \quad\text{and}\quad
	\exists i^\ast\in I,\ \SCIG(\mathcal P_{i^\ast})=k.
	\]
	
	\item \textbf{Worst-case exactness at \(k\):}
	We say that \(\mathcal F\) is worst-case exact at \(k\) if
	\[
	\sup_{i\in I}\SCIG(\mathcal P_i)=k.
	\]
\end{enumerate}
\end{definition}

The three family-level notions isolate three different logical strengths.
Family-pointwise exactness is a uniform statement: every member of the family has the same exact height $k$.
Witness-space sharpness is weaker: it consists of a uniform upper bound by $k$ together with a single family member that attains height $k$.
Worst-case exactness is the equivalent supremum reformulation of the same phenomenon. Thus the upgrade problem studied in this paper is exactly the problem of turning one
extremal witness into a uniform lower bound for the whole family.

\begin{remark}[Parametrized sets as families]
Whenever a class of problems is written in set form
\[
\mathcal F=\{\mathcal P_\xi:\xi\in J\},
\]
we identify it with the indexed family $(\mathcal P_\xi)_{\xi\in J}$.
\end{remark}

\begin{lemma}[Witness-space sharpness \(=\) worst-case exactness]{\label{lem:WitWorstEq}}
Let \(\mathcal F=(\mathcal P_i)_{i\in I}\) and let \(k\in\N_0\). Then the following are equivalent:
\begin{enumerate}[label=(\alph*)]
	\item \(\mathcal F\) is witness-space sharp at \(k\);
	\item \(\mathcal F\) is worst-case exact at \(k\).
\end{enumerate}
\end{lemma}
\begin{proof}
If \(\mathcal F\) is witness-space sharp at \(k\), then \(\SCIG(\mathcal P_i)\leq k\) for all \(i\) and \(\SCIG(\mathcal P_{i^\ast})=k\) for some \(i^\ast\), hence \(\sup_i \SCIG(\mathcal P_i)=k\).

Conversely, if \(\sup_i \SCIG(\mathcal P_i)=k<\infty\), then automatically \(\SCIG(\mathcal P_i)\leq k\) for all \(i\). Since the range of \(\SCIG\) is the well-ordered set \(\N_0\cup\{\infty\}\), the equality \(\sup_i\SCIG(\mathcal P_i)=k\in \N_0\) implies that \(k\) is attained by some \(i^\ast\), hence \(\SCIG(\mathcal P_{i^\ast})=k\).
\end{proof}

\begin{proposition}[Logical relation between the three notions]{\label{prop:LogRelNot}}
For every family \(\mathcal F=(\mathcal P_i)_{i\in I}\) and every \(k\in\N_0\),
\begin{align*}
	\text{family-pointwise exactness at }k \Longrightarrow
	\text{witness-space sharpness at }k \\ \Longleftrightarrow
	\text{worst-case exactness at }k.
\end{align*}
The reverse implication from witness-space sharpness to family-pointwise exactness is false in general.
\end{proposition}
\begin{proof}
The equivalence between witness-space sharpness and worst-case exactness is exactly \cref{lem:WitWorstEq}; the forward implication from family-pointwise exactness to witness-space sharpness follows directly from \cref{def:sharpNot}; the failure of the reverse implication is proved by \cref{cor:source-family-strictness} later.
\end{proof}

\section{A Canonical Strictness Example}\label{sec:strictness-source-family}

The purpose of this section is to show, without using any external operator-theoretic
classification, that witness-space sharpness is strictly weaker than family-pointwise
exactness.  The example is deliberately elementary: it is a one-limit coordinate-search
problem together with a height-zero degenerate problem.

\begin{definition}[The coordinate-existence source problem]
\label{def:coordinate-existence-source}
Let
\[
        \Omega_{\exists}:=\{0,1\}^{\mathbb N}
\]
with coordinate evaluations
\[
        \lambda_n:\Omega_{\exists}\to\mathbb C,
        \qquad
        \lambda_n(a):=a_n
\]
with $n\in\mathbb N$. Set
\[
        \Lambda_{\exists}:=\{\lambda_n:n\in\mathbb N\}.
\]
Define
\[
        \Xi_{\exists}:\Omega_{\exists}\to\{0,1\}
\]
by
\[
        \Xi_{\exists}(a)=1
        \quad\Longleftrightarrow\quad
        \bigl( \exists n\in\mathbb N\bigr) \ a_n=1.
\]
Equip \(\{0,1\}\) with the discrete metric
\[
        d_{\mathrm{disc}}(u,v):=
        \begin{cases}
        0,&u=v,\\
        1,&u\ne v.
        \end{cases}
\]
Finally define then
\[
        \mathcal P_{\exists}:= (\Xi_{\exists},\Omega_{\exists},(\{0,1\},d_{\mathrm{disc}}),\Lambda_{\exists}).
\]
\end{definition}

\begin{lemma}\label{lem:P-exists-well-defined}
\(\mathcal P_{\exists}\) is an SCI computational problem.
\end{lemma}

\begin{proof}
We only need to check the consistency condition. Let \(a,b\in\Omega_{\exists}\) and assume
\[
        \Xi_{\exists}(a)\ne\Xi_{\exists}(b).
\]
Without loss of generality,
\[
        \Xi_{\exists}(a)=1, \quad
        \Xi_{\exists}(b)=0.
\]
Then there is an \(n\in\mathbb N\) with \(a_n=1\). Since \(\Xi_{\exists}(b)=0\), one has \(b_j=0\) for every \(j\), in particular \(b_n=0\). Hence
\[
        \lambda_n(a)=1\ne0=\lambda_n(b).
\]
Thus the evaluation family separates every pair of inputs with different target values.
\end{proof}

\begin{proposition}[The coordinate-existence source has exact type-G height one]\label{prop:P-exists-height-one}
One has
\[
        \mathrm{SCI}_G(\mathcal P_{\exists})=1.
\]
\end{proposition}

\begin{proof}
This is the \(m=1\) instance of the canonical Cantor-matrix source family \(\{\mathcal P_m^{\mathrm{mat}}\}_{m\in \mathbb{N}}\) from \cite[Sec.~9]{Sor26}. Indeed: after a harmless reindexing of the coordinate oracle, \(\mathcal P_\exists\) is equivalent to \(\mathcal P_1^{\mathrm{mat}}\). Hence
\[
        \SCIG(\mathcal P_\exists)=1
\]
by \cite[Thm.~9.7]{Sor26}.
\end{proof}

\begin{definition}[A height-zero degenerate decision problem]
\label{def:height-zero-degenerate}
Let
\[
        \Omega_0:=\{\ast\},
        \qquad
        \Lambda_0:=\{c\},
        \qquad
        c(\ast):=0,
\]
and define
\[
        \Xi_0:\Omega_0\to\{0,1\},
        \qquad
        \Xi_0(\ast):=0.
\]
Set
\[
        \mathcal P_0:=(\Xi_0,\Omega_0,(\{0,1\},d_{\mathrm{disc}}),\Lambda_0).
\]
\end{definition}

\begin{lemma}\label{lem:P0-height-zero}
One has
\[
        \mathrm{SCI}_G(\mathcal P_0)=0.
\]
\end{lemma}

\begin{proof}
Define
\[
        \Gamma_0(\ast):=0, \quad \Lambda_{\Gamma_0}(\ast):=\{c\}.
\]
Since there is only one input, the defining stability condition for a general algorithm is automatic. Also
\[
        \Gamma_0=\Xi_0.
\]
Thus \(\mathcal P_0\) is solved by a single height-\(0\) general algorithm, and hence
\[
        \mathrm{SCI}_G(\mathcal P_0)=0.
\]
\end{proof}

\begin{definition}[The strictness family]\label{def:strictness-family}
Define the two-member family
\[
        \mathcal F_{\mathrm{str}} := (\mathcal P_0, \mathcal P_{\exists}).
\]
\end{definition}

\begin{corollary}[Witness-space sharpness need not be family-pointwise exact]\label{cor:source-family-strictness}
The family \(\mathcal F_{\mathrm{str}}\) is witness-space sharp, equivalently worst-case exact, at height \(1\), but it is not family-pointwise exact at height \(1\).
\end{corollary}

\begin{proof}
By \cref{lem:P0-height-zero},
\[
        \mathrm{SCI}_G(\mathcal P_0)=0,
\]
and by \cref{prop:P-exists-height-one},
\[
        \mathrm{SCI}_G(\mathcal P_{\exists})=1.
\]
Therefore every member of \(\mathcal F_{\mathrm{str}}\) has type-G SCI at most \(1\), and one member, namely \(\mathcal P_{\exists}\), has exact type-G SCI equal to \(1\). Thus \(\mathcal F_{\mathrm{str}}\) is witness-space sharp at height \(1\). By \cref{lem:WitWorstEq}, witness-space sharpness is equivalent to worst-case exactness, so \(\mathcal F_{\mathrm{str}}\) is also worst-case exact at height \(1\).

However,
\[
        \mathrm{SCI}_G(\mathcal P_0)=0\ne1.
\]
Thus not every member of \(\mathcal F_{\mathrm{str}}\) has exact height \(1\), so the family is not family-pointwise exact at height \(1\).
\end{proof}

\begin{corollary}[No generic upgrade theorem]\label{cor:no-generic-upgrade}
There is no theorem, valid for arbitrary nonempty families of type-G SCI problems, of the schematic form
\[
        \text{witness-space sharpness at }k
        \quad\Longrightarrow\quad
        \text{family-pointwise exactness at }k.
\]
Equivalently, a universal upper bound together with one exact witness does not, by itself, imply a uniform lower bound on every family member.
\end{corollary}

\begin{proof}
The family \(\mathcal F_{\mathrm{str}}\) from \cref{def:strictness-family} is witness-space sharp at height \(1\) by \cref{cor:source-family-strictness}, but it is not family-pointwise exact at height \(1\).  Hence the implication fails already for \(k=1\).
\end{proof}

A natural next question is now: when do we actually have equivalence of witness-space sharpness and family-pointwise exactness? \Cref{sec:AbstrUpgrTh} develops abstract sufficient criteria and their decoder-regular transport refinement, \cref{sec:interinteupgrade} gives two natural positive families, and \cref{sec:LattTranspQuot} analyzes the transport-degree structure for the two natural decoder classes \(\mathscr R_{\mathrm{cont}}\) and \(\mathscr R_{\mathrm{Bor}}\).

\section{Abstract Upgrade Theorems}{\label{sec:AbstrUpgrTh}}
We start by explaining and proving sufficient abstract premises that one needs to ensure this equivalence. The idea is the following:
Suppose there is a fixed source problem $\mathcal S_k$ of exact height $k$. If every hypothetical height-$(k-1)$ solution of a family member $\mathcal P_i$
can be transported back to a height-$(k-1)$ solution of $\mathcal S_k$, then no family member can in fact have height at most $k-1$.
The next definitions isolate this mechanism abstractly, and the finite-query reduction criterion below gives one concrete way to realize it.

\begin{definition}[Height-\((k-1)\) pullback property relative to a source problem]\label{def:pullback-property}
Fix \(k\in\mathbb N\). Let
\[
\mathcal S_k=(\Psi_k,\Sigma_k,(\mathcal N_k,\rho_k),\Lambda_k)
\]
be an SCI computational problem, and let
\[
\mathcal F=(\mathcal P_i)_{i\in I}, \quad
\mathcal P_i=(\Xi_i,\Omega_i,(\mathcal M_i,d_i),\Lambda_i).
\]
We say that $\mathcal F$ has the height-$(k-1)$ pullback property relative to $\mathcal S_k$
if for every $i\in I$ there exists a map
\[
\mathfrak P_i:
\{\text{height-}(k-1)\text{ type-}G\text{ towers on } \mathcal P_i\} \to
\{\text{height-}(k-1)\text{ type-}G\text{ towers on } \mathcal S_k\}
\]
which assigns to each height-$(k-1)$ type-$G$ tower on $\mathcal P_i$ a height-$(k-1)$ type-$G$ tower on $\mathcal S_k$.

In particular, whenever some $\mathcal P_i$ admits a height-$(k-1)$ type-$G$ tower, the source problem $\mathcal S_k$ also admits one.
\end{definition}

\begin{definition}[Abstract hypotheses \((H1_k)\)-\((H3_k)\)]\label{def:abstract-Hk}
Fix \(k\in\mathbb N\) and let \(\mathcal F=(\mathcal P_i)_{i\in I}\) be a family of SCI computational problems.
We say that \(\mathcal F\) satisfies the \textit{abstract level-\(k\) hypotheses}
if there exists a fixed source problem
\[
\mathcal S_k=(\Psi_k,\Sigma_k,(\mathcal N_k,\rho_k),\Lambda_k)
\]
such that
\begin{enumerate}[label=(H\arabic*$_k$)]
	\item \label{H1-abstract}
	\(\SCIG(\mathcal S_k)=k\);
	
	\item \label{H2-abstract}
	\(\mathcal F\) has the height-\((k-1)\) pullback property relative to \(\mathcal S_k\)
	in the sense of \cref{def:pullback-property};
	
	\item \label{H3-abstract}
	\(\mathrm{UB}_k(\mathcal F)\) holds, i.e.
	\[
	\forall i\in I,\qquad \SCIG(\mathcal P_i)\le k.
	\]
\end{enumerate}
\end{definition}

\begin{theorem}[Abstract lower-bound transfer]\label{thm:abstract-lower-bound-transfer}
Assume \((H1_k)\) and \((H2_k)\) from \cref{def:abstract-Hk}.
Then
\[
\mathrm{LB}_k(\mathcal F) \quad\text{holds, i.e.}\quad
\forall i\in I,\ \SCIG(\mathcal P_i)\ge k.
\]
\end{theorem}

\begin{proof}
Fix \(i\in I\).
Assume, towards a contradiction, that \(\SCIG(\mathcal P_i)\le k-1\).
Then by definition of \(\SCIG\) there exists a height-\((k-1)\) type-\(G\) tower computing \(\Xi_i\) on \(\mathcal P_i\).
By \((H2_k)\), this tower pulls back to a height-\((k-1)\) type-\(G\) tower on \(\mathcal S_k\).
Hence
\[
\SCIG(S_k)\le k-1,
\]
contradicting \((H1_k)\), namely \(\SCIG(\mathcal S_k)=k\). Therefore \(\SCIG(\mathcal P_i)\ge k\).
Since \(i\in I\) was arbitrary, \(\mathrm{LB}_k(\mathcal F)\) holds.
\end{proof}

\begin{corollary}[Abstract sufficiency for family-pointwise exactness]\label{cor:abstract-sufficiency-pointwise-exact}
Assume \((H1_k)\), \((H2_k)\), and \((H3_k)\) from \cref{def:abstract-Hk}.
Then
\[
\forall i\in I,\qquad \SCIG(\mathcal P_i)=k.
\]
In other words, \(\mathcal F\) is family-pointwise exact at \(k\).
\end{corollary}

\begin{proof}
By \cref{thm:abstract-lower-bound-transfer}, \((H1_k)\) and \((H2_k)\) imply
\(\mathrm{LB}_k(\mathcal F)\).
By \((H3_k)\), we already have \(\mathrm{UB}_k(\mathcal F)\).
Hence
\[
\forall i\in I,\qquad k\le \SCIG(\mathcal P_i)\le k,
\]
so \(\SCIG(\mathcal P_i)=k\) for all \(i\in I\).
\end{proof}

\begin{corollary}[Equivalence of the three family-level sharpness notions under \((H1_k)\) and \((H2_k)\)]\label{cor:equivalence-under-pullback}
Assume \((H1_k)\) and \((H2_k)\) from \cref{def:abstract-Hk}.
Then the following are equivalent
\begin{enumerate}[label=(\roman*)]
	\item \(\mathcal F\) is family-pointwise exact at \(k\);
	\item \(\mathcal F\) is witness-space sharp at \(k\);
	\item \(\mathcal F\) is worst-case exact at \(k\).
\end{enumerate}
If moreover \((H3_k)\) holds, then all three properties hold.
\end{corollary}

\begin{proof}
The chain
\[
\text{family-pointwise exactness at }k \Longrightarrow
\text{witness-space sharpness at }k \Longleftrightarrow
\text{worst-case exactness at }k
\]
is already proven within \cref{lem:WitWorstEq} and \cref{prop:LogRelNot}.

It remains to show that (ii) or (iii) imply (i) under \((H1_k)\) and \((H2_k)\).
By \cref{thm:abstract-lower-bound-transfer}, \((H1_k)\) and \((H2_k)\) give
\[
\forall i\in I,\qquad \SCIG(\mathcal P_i)\ge k.
\]

If (ii) holds, then by definition of witness-space sharpness we also have
\[
\forall i\in I,\qquad \SCIG(\mathcal P_i)\le k.
\]
Hence \(\SCIG(\mathcal P_i)=k\) for all \(i\), i.e. (i) holds.

If (iii) holds, then
\[
\sup_{i\in I}\SCIG(\mathcal P_i)=k.
\]
Since every \(\SCIG(\mathcal P_i)\ge k\), the supremum can equal \(k\) only if
\[
\SCIG(\mathcal P_i)=k\qquad\text{for every }i\in I.
\]
Thus (i) holds again. The final sentence follows from \cref{cor:abstract-sufficiency-pointwise-exact}.
\end{proof}

\begin{remark}[What the theorem does \textit{not} claim]\label{rem:not-necessary}
The hypotheses above are sufficient, \textit{not} necessary.
Nothing in \cref{thm:abstract-lower-bound-transfer} or \cref{cor:concrete-sufficiency-package} asserts that every family-pointwise exactness theorem must arise from such a common source problem or from a finite-query evaluation reduction.
The open problem is precisely to identify natural \textit{necessary and sufficient} criteria on a family \(\mathcal F\) for the implication
\[
\text{witness-space sharpness at }k \Longrightarrow
\text{family-pointwise exactness at }k.
\]
\end{remark}

\begin{remark}[What fails in the strictness family]\label{rem:what-fails-strictness-family}
The strictness family
\[
        \mathcal F_{\mathrm{str}}=(\mathcal P_0, \mathcal P_{\exists})
\]
from \cref{def:strictness-family} shows why some extra hypothesis is necessary in any upgrade theorem from witness-space sharpness to family-pointwise exactness. Consequently, there cannot exist any source
problem \(\mathcal S_1\) with
\[
        \mathrm{SCI}_G(\mathcal S_1)=1
\]
such that the height-\(0\) pullback property from \cref{def:pullback-property} holds for every member of \(\mathcal F_{\mathrm{str}}\).
\end{remark}

The previous discussion shows that witness-space sharpness alone does not supply the family-wide pullback mechanism. The next step is therefore to isolate a concrete transport condition strong enough to pull lower bounds back from every family member to a common source problem. In the exact-input setting, finite-query evaluation reduction is the most direct such mechanism.

\begin{definition}[Finite-query evaluation reduction]\label{def:fq-eval-reduction}
Let
\[
\mathcal S=(\Psi,\Sigma,(\mathcal N,\rho),\Lambda_{\mathcal S}), \quad
\mathcal P=(\Xi,\Omega,(\mathcal M,d),\Lambda_{\mathcal P})
\]
be SCI computational problems.
We write
\[
\mathcal S\leq_{G,\mathrm{fq}} \mathcal P
\]
and say that \textit{\(\mathcal S\) finitely evaluation-reduces to \(\mathcal P\)} if there exists
\begin{enumerate}[label=(\alph*)]
	\item an encoding map
	\[
	E:\Sigma\to\Omega;
	\]
	
	\item a continuous decoding map
	\[
	D:(\mathcal M,d)\to (\mathcal N,\rho);
	\]
	
	\item for every \(f\in\Lambda_{\mathcal P}\), a natural number \(m_f\in \mathbb{N}\), source evaluations
	\[
	\gamma_{f,1},\dots,\gamma_{f,m_f}\in \Lambda_{\mathcal S},
	\]
	and a map
	\[
	\vartheta_f:\operatorname{im}(\gamma_{f,1},\dots,\gamma_{f,m_f})\to\mathbb C
	\]
	such that for every \(A\in\Sigma\),
	\[
	\Psi(A)=D(\Xi(E(A)))
	\]
	and
	\[
	f(E(A)) =
	\vartheta_f\bigl(\gamma_{f,1}(A),\dots,\gamma_{f,m_f}(A)\bigr).
	\]
\end{enumerate}
\end{definition}

\begin{lemma}[Pullback of a general algorithm along a finite-query evaluation reduction]\label{lem:pullback-general-algorithm}
Assume \(\mathcal S\leq_{G,\mathrm{fq}} \mathcal P\) as in \cref{def:fq-eval-reduction}.
Let \((\Gamma_{\mathcal P},\Lambda_{\Gamma_{\mathcal P}})\) be a general algorithm on \(\mathcal P\) in the sense of \cref{def:GenAlgoSCI}.
Define
\[
\Gamma_{\mathcal S}(A):=D\bigl(\Gamma_{\mathcal P}(E(A))\bigr),
\]
and
\[
\Lambda_{\Gamma_{\mathcal S}}(A) :=
\bigcup_{f\in \Lambda_{\Gamma_{\mathcal P}}(E(A))}
\{\gamma_{f,1},\dots,\gamma_{f,m_f}\}.
\]
Then \((\Gamma_{\mathcal S},\Lambda_{\Gamma_{\mathcal S}})\) is a general algorithm on \(\mathcal S\).
\end{lemma}

\begin{proof}
Fix \(A,B\in\Sigma\) and assume that
\[
\eta(A)=\eta(B)\qquad\text{for all }\eta\in \Lambda_{\Gamma_{\mathcal S}}(A).
\]
We show that this implies
\[
\Gamma_{\mathcal S}(A)=\Gamma_{\mathcal S}(B) \quad\text{and}\quad
\Lambda_{\Gamma_{\mathcal S}}(A)=\Lambda_{\Gamma_{\mathcal S}}(B).
\]
Take any
\[
f\in \Lambda_{\Gamma_{\mathcal P}}(E(A)).
\]
By definition of \(\Lambda_{\Gamma_{\mathcal S}}(A)\), all evaluations
\(\gamma_{f,1},\dots,\gamma_{f,m_f}\) belong to \(\Lambda_{\Gamma_{\mathcal S}}(A)\).
Hence
\[
\gamma_{f,j}(A)=\gamma_{f,j}(B)\qquad (j=1,\dots,m_f).
\]
Therefore, by the reconstruction formula in \cref{def:fq-eval-reduction},
\[
f(E(A)) =
\vartheta_f\bigl(\gamma_{f,1}(A),\dots,\gamma_{f,m_f}(A)\bigr) =
\vartheta_f\bigl(\gamma_{f,1}(B),\dots,\gamma_{f,m_f}(B)\bigr) = f(E(B)).
\]
So every query asked by \(\Gamma_{\mathcal P}\) at input \(E(A)\) takes the same value on \(E(B)\).
Since \((\Gamma_{\mathcal P},\Lambda_{\Gamma_{\mathcal P}})\) is a general algorithm on \(\mathcal P\), it follows that
\[
\Gamma_{\mathcal P}(E(A))=\Gamma_{\mathcal P}(E(B)) \quad\text{and}\quad
\Lambda_{\Gamma_{\mathcal P}}(E(A))=\Lambda_{\Gamma_{\mathcal P}}(E(B)).
\]
Applying the continuous decoder \(D\), we obtain
\[
\Gamma_{\mathcal S}(A)=D(\Gamma_{\mathcal P}(E(A)))=D(\Gamma_{\mathcal P}(E(B)))=\Gamma_{\mathcal S}(B).
\]
Moreover, because the queried set \(\Lambda_{\Gamma_P}(E(A))\) agrees with
\(\Lambda_{\Gamma_P}(E(B))\), the unions of the associated source-query blocks agree, i.e.
\[
\Lambda_{\Gamma_{\mathcal S}}(A)=\Lambda_{\Gamma_{\mathcal S}}(B).
\]
Thus \((\Gamma_{\mathcal S},\Lambda_{\Gamma_{\mathcal S}})\) satisfies the defining condition of a general
algorithm on \(\mathcal S\).
\end{proof}

\begin{theorem}[SCI monotonicity under finite-query evaluation reductions]\label{thm:SCI-monotonicity-fq-reduction}
If
\[
\mathcal S\leq_{G,\mathrm{fq}} \mathcal P,
\]
then
\[
\SCIG(\mathcal S)\le \SCIG(\mathcal P).
\]
\end{theorem}

\begin{proof}
If \(\SCIG(\mathcal P)=\infty\), there is nothing to prove. Let
\[
        \SCIG(\mathcal P)=m<\infty.
\]
If \(m=0\), let \((\Gamma_\mathcal P,\Lambda_{\Gamma_\mathcal P})\) be a height-zero general algorithm on \(\mathcal P\), so that \(\Gamma_\mathcal P=\Xi\). Define
\[
        \Gamma_\mathcal S(A) := D(\Gamma_\mathcal P(E(A)))
\]
for $A\in\Sigma$. By \cref{lem:pullback-general-algorithm}, \((\Gamma_\mathcal S,\Lambda_{\Gamma_\mathcal S})\) is a general algorithm on \(\mathcal S\). Moreover, by the target identity in \cref{def:fq-eval-reduction},
\[
        \Gamma_\mathcal S(A) = D(\Xi(E(A))) =
        \Psi(A).
\]
Hence \(\SCIG(\mathcal S)=0\le \SCIG(\mathcal P)\).

Now assume \(m\ge1\). Let
\[
        \Gamma_{n_m,\ldots,n_1}:\Omega\to \mathcal M,
\]
where $(n_1,\ldots,n_m)\in\mathbb N^m$, be the deepest-level general algorithms of a standard raw type-\(G\) tower of height \(m\) for \(\mathcal P\). Thus, for every \(B\in\Omega\),
\[
        \Xi(B)= \lim_{n_m\to\infty} \lim_{n_{m-1}\to\infty} \cdots \lim_{n_1\to\infty} \Gamma_{n_m,\ldots,n_1}(B).
\]
For each multi-index \((n_m,\ldots,n_1)\), define
\[
        \widetilde\Gamma_{n_m,\ldots,n_1}(A) :=
        D\bigl(\Gamma_{n_m,\ldots,n_1}(E(A))\bigr),
\]
where $A\in\Sigma$. By \cref{lem:pullback-general-algorithm}, each deepest-level map
\[
        \widetilde\Gamma_{n_m,\ldots,n_1}
\]
is a general algorithm on \(\mathcal S\). It remains only to check the iterated limits. Fix \(A\in\Sigma\). Since \(D\) is continuous and the tower limits for \(\mathcal P\) exist, continuity of \(D\) gives,
\[
\begin{aligned}
&\lim_{n_m\to\infty}\cdots\lim_{n_1\to\infty}
        \widetilde\Gamma_{n_m,\ldots,n_1}(A)  \\
&\quad =
        \lim_{n_m\to\infty}\cdots\lim_{n_1\to\infty}
        D\bigl(\Gamma_{n_m,\ldots,n_1}(E(A))\bigr)  \\
&\quad =
        D\left(
        \lim_{n_m\to\infty}\cdots\lim_{n_1\to\infty}
        \Gamma_{n_m,\ldots,n_1}(E(A))
        \right)  \\
&\quad =
        D(\Xi(E(A))) = \Psi(A).
\end{aligned}
\]
The interchange of \(D\) with the iterated limits is justified one limit at a time by ordinary sequential continuity in metric spaces.

Thus the deepest-level maps $\widetilde\Gamma_{n_m,\ldots,n_1}$ form a standard raw type-\(G\) tower of height \(m\) for \(\mathcal S\). Therefore
\[
        \SCIG(\mathcal S)\le m=\SCIG(\mathcal P).
\]
\end{proof}

\begin{corollary}[A concrete sufficient criterion for \((H2_k)\)]\label{cor:concrete-criterion-for-H2}
Let \(\mathcal F=(\mathcal P_i)_{i\in I}\) and let \(\mathcal S_k\) be a fixed SCI computational problem.
Assume that for every \(i\in I\),
\[
\mathcal S_k\leq_{G,\mathrm{fq}} \mathcal P_i.
\]
Then \(\mathcal F\) has the height-\((k-1)\) pullback property relative to \(\mathcal S_k\).
Equivalently, the family-wide finite-query evaluation reductions imply \((H2_k)\).
\end{corollary}

\begin{proof}
Fix $i\in I$. Given a height-$(k-1)$ type-$G$ tower on $\mathcal P_i$, apply the construction from \cref{thm:SCI-monotonicity-fq-reduction} to obtain its pulled-back
height-$(k-1)$ tower on $\mathcal S_k$. This defines the map $\mathfrak P_i$ required in \cref{def:pullback-property}.
Since $i$ was arbitrary, $\mathcal F$ has the height-$(k-1)$ pullback property relative to $\mathcal S_k$.
\end{proof}

\begin{corollary}[Concrete sufficiency package]\label{cor:concrete-sufficiency-package}
Fix \(k\in\mathbb N\). Let \(\mathcal F=(\mathcal P_i)_{i\in I}\) be a family of SCI computational problems.
Assume there exists a fixed source problem \(\mathcal S_k\) such that
\begin{enumerate}[label=(C\arabic*)]
	\item \(\SCIG(\mathcal S_k)=k\);
	\item for every \(i\in I\), one has \(\mathcal S_k\leq_{G,\mathrm{fq}} \mathcal P_i\);
	\item \(\mathrm{UB}_k(\mathcal F)\) holds.
\end{enumerate}
Then \(\mathcal F\) is family-pointwise exact at \(k\). Moreover, under (C1)-(C2), the three family-level sharpness notions are equivalent, i.e. family-pointwise exactness at \(k\), witness-space sharpness at \(k\), and worst-case exactness at \(k\).
\end{corollary}

\begin{proof}
By (C2) and \cref{cor:concrete-criterion-for-H2}, \((H2_k)\) holds. Now apply \cref{cor:abstract-sufficiency-pointwise-exact} and \cref{cor:equivalence-under-pullback}.
\end{proof}

\subsection{A Decoder-Regular Finite-Query Transport Preorder}\label{sec:decoder-regular-transport}

\begin{definition}[Admissible decoder class]\label{def:admissible-decoder-class}
A class \(\mathscr R\) of maps between metric spaces is called \textit{admissible} if:
\begin{enumerate}
	\item for every metric space \((\mathcal M,d)\), the identity map
	\[
	\mathrm{id}_{\mathcal M}:(\mathcal M,d)\to (\mathcal M,d)
	\]
	belongs to \(\mathscr R\);
	\item whenever
	\[
	f:(\mathcal M,d)\to (\mathcal N,\rho), \quad
	g:(\mathcal N,\rho)\to (\mathcal L,\lambda)
	\]
	belong to \(\mathscr R\), then the composition
	\[
	g\circ f:(\mathcal M,d)\to (\mathcal L,\lambda)
	\]
	also belongs to \(\mathscr R\).
\end{enumerate}

The two basic examples are
\[
\mathscr R_{\mathrm{cont}} :=
\{\text{continuous maps between metric spaces}\}, \quad \mathscr R_{\mathrm{Bor}} :=
\{\text{Borel measurable maps between metric spaces}\}.
\]
\end{definition}

\begin{definition}[Decoder-regular finite-query transport]\label{def:decoder-regular-fq-transport}
Let
\[
\mathcal Q=(\Psi,\Sigma,(\mathcal N,\rho),\Gamma), \quad
\mathcal P=(\Xi,\Omega,(\mathcal M,d),\Lambda)
\]
be SCI computational problems, and let \(\mathscr R\) be an admissible decoder class. We write
\[
\mathcal Q \leq^{\mathscr R}_{\mathrm{fq}} \mathcal P
\]
and say that \(\mathcal Q\) \textit{\(\mathscr R\)-decoder finitely transports to \(\mathcal P\)} if there exists
\begin{enumerate}
	\item an encoding map
	\[
	E:\Sigma\to\Omega;
	\]
	\item a decoder
	\[
	D:(\mathcal M,d)\to (\mathcal N,\rho)
	\]
	with \(D\in \mathscr R\);
	\item for every \(f\in\Lambda\), a natural number \(m_f\in\mathbb N\), source evaluations
	\[
	\gamma_{f,1},\dots,\gamma_{f,m_f}\in\Gamma,
	\]
	and a map
	\[
	\vartheta_f:\operatorname{im}(\gamma_{f,1},\dots,\gamma_{f,m_f})\to\mathbb C
	\]
	such that for every \(A\in\Sigma\),
	\[
	\Psi(A)=D(\Xi(E(A)))
	\]
	and
	\[
	f(E(A)) =
	\vartheta_f\bigl(\gamma_{f,1}(A),\dots,\gamma_{f,m_f}(A)\bigr).
	\]
\end{enumerate}
\end{definition}

\begin{proposition}[The relation \(\leq^{\mathscr R}_{\mathrm{fq}}\) is a preorder]\label{prop:decoder-regular-fq-preorder}
For every admissible decoder class \(\mathscr R\), the relation $\leq^{\mathscr R}_{\mathrm{fq}}$ on SCI computational problems is reflexive and transitive.
\end{proposition}

\begin{proof}
Let
\[
\mathcal R=(\Theta,\Pi,(\mathcal L,\lambda),\Delta), \quad
\mathcal Q=(\Psi,\Sigma,(\mathcal N,\rho),\Gamma), \quad
\mathcal P=(\Xi,\Omega,(\mathcal M,d),\Lambda)
\]
be SCI computational problems.

For reflexivity, fix \(\mathcal P\). Take
\[
E:=\mathrm{id}_\Omega, \quad D:=\mathrm{id}_{\mathcal M}.
\]
By admissibility, \(D\in\mathscr R\). For each \(f\in\Lambda\), set
\[
m_f:=1, \quad \gamma_{f,1}:=f, \quad \vartheta_f(z):=z.
\]
Then for every \(A\in\Omega\),
\[
\Xi(A)=D(\Xi(E(A)))
\]
and
\[
f(E(A)) = f(A) = \vartheta_f(\gamma_{f,1}(A)).
\]
Hence
\[
\mathcal P \leq^{\mathscr R}_{\mathrm{fq}} \mathcal P.
\]

For transitivity, assume
\[
\mathcal R \leq^{\mathscr R}_{\mathrm{fq}} \mathcal Q \quad\text{and}\quad
\mathcal Q \leq^{\mathscr R}_{\mathrm{fq}} \mathcal P.
\]
Choose witnessing data for
\[
\mathcal R \leq^{\mathscr R}_{\mathrm{fq}} \mathcal Q
\]
as follows.
\[
E_{\mathcal{RQ}}:\Pi\to\Sigma, \quad
D_{\mathcal{RQ}}:(\mathcal N,\rho)\to (\mathcal L,\lambda),
\]
and for every \(g\in\Gamma\), numbers \(m_g\in\mathbb N\), source evaluations
\[
\delta_{g,1},\dots,\delta_{g,m_g}\in\Delta,
\]
and maps
\[
\theta_g:\operatorname{im}(\delta_{g,1},\dots,\delta_{g,m_g})\to\mathbb C
\]
such that for every \(A\in\Pi\),
\[
\Theta(A)=D_{\mathcal{RQ}}(\Psi(E_{\mathcal{RQ}}(A)))
\]
and
\[
g(E_{\mathcal{RQ}}(A)) = \theta_g\bigl(\delta_{g,1}(A),\dots,\delta_{g,m_g}(A)\bigr).
\]

Choose witnessing data for
\[
\mathcal Q \leq^{\mathscr R}_{\mathrm{fq}} \mathcal P
\]
as follows.
\[
E_{\mathcal{QP}}:\Sigma\to\Omega, \quad
D_{\mathcal{QP}}:(\mathcal M,d)\to (\mathcal N,\rho),
\]
and for every \(f\in\Lambda\), numbers \(n_f\in\mathbb N\), source evaluations
\[
\gamma_{f,1},\dots,\gamma_{f,n_f}\in\Gamma,
\]
and maps
\[
\vartheta_f:\operatorname{im}(\gamma_{f,1},\dots,\gamma_{f,n_f})\to\mathbb C
\]
such that for every \(B\in\Sigma\),
\[
\Psi(B)=D_{\mathcal{QP}}(\Xi(E_{\mathcal{QP}}(B)))
\]
and
\[
f(E_{\mathcal{QP}}(B)) = \vartheta_f\bigl(\gamma_{f,1}(B),\dots,\gamma_{f,n_f}(B)\bigr).
\]

Define the composed encoding
\[
E_{\mathcal{RP}}:=E_{\mathcal{QP}}\circ E_{\mathcal{RQ}}:\Pi\to\Omega
\]
and the composed decoder
\[
D_{\mathcal{RP}}:=D_{\mathcal{RQ}}\circ D_{\mathcal{QP}}:(\mathcal M,d)\to(\mathcal L,\lambda).
\]
Since \(\mathscr R\) is admissible and both \(D_{\mathcal{RQ}}\) and \(D_{\mathcal{QP}}\) belong to \(\mathscr R\), we have
\[
D_{\mathcal{RP}}\in\mathscr R.
\]

Now fix \(f\in\Lambda\). For each \(j\in\{1,\dots,n_f\}\), the source query \(\gamma_{f,j}\in\Gamma\) can be transported along
\[
\mathcal R \leq^{\mathscr R}_{\mathrm{fq}} \mathcal Q.
\]
Thus there exist a natural number \(m_{f,j}\in\mathbb N\), source evaluations
\[
\delta_{f,j,1},\dots,\delta_{f,j,m_{f,j}}\in\Delta,
\]
and a map
\[
\theta_{f,j}: \operatorname{im}(\delta_{f,j,1},\dots,\delta_{f,j,m_{f,j}})\to\mathbb C
\]
such that for every \(A\in\Pi\),
\[
\gamma_{f,j}(E_{\mathcal{RQ}}(A)) = \theta_{f,j}\bigl(\delta_{f,j,1}(A),\dots,\delta_{f,j,m_{f,j}}(A)\bigr).
\]

Let
\[
N_f:=\sum_{j=1}^{n_f} m_{f,j}.
\]
Choose an enumeration
\[
\delta'_{f,1},\dots,\delta'_{f,N_f}
\]
of the concatenated family
\[
(\delta_{f,1,1},\dots,\delta_{f,1,m_{f,1}}, \dots, \delta_{f,n_f,1},\dots,\delta_{f,n_f,m_{f,n_f}}).
\]
Define
\[
\vartheta'_f: \operatorname{im}(\delta'_{f,1},\dots,\delta'_{f,N_f})\to\mathbb C
\]
by blockwise substitution, i.e.
\begin{align*}
	\vartheta'_f \bigl( z_{1,1},\dots,z_{1,m_{f,1}}, \dots, z_{n_f,1},\dots,z_{n_f,m_{f,n_f}} \bigr) := \\
	\vartheta_f\Bigl( \theta_{f,1}(z_{1,1},\dots,z_{1,m_{f,1}}), \dots, \theta_{f,n_f}(z_{n_f,1},\dots,z_{n_f,m_{f,n_f}}) \Bigr).
\end{align*}

Then for every \(A\in\Pi\),
\begin{align*}
	f(E_{\mathcal{RP}}(A))
	&= f(E_{\mathcal{QP}}(E_{\mathcal{RQ}}(A)))\\
	&= \vartheta_f\bigl( \gamma_{f,1}(E_{\mathcal{RQ}}(A)),\dots,\gamma_{f,n_f}(E_{\mathcal{RQ}}(A)) \bigr)\\
	&= \vartheta_f\Bigl( \theta_{f,1}(\delta_{f,1,1}(A),\dots,\delta_{f,1,m_{f,1}}(A)), \dots, \theta_{f,n_f}(\delta_{f,n_f,1}(A),\dots,\delta_{f,n_f,m_{f,n_f}}(A)) \Bigr)\\
	&= \vartheta'_f\bigl(\delta'_{f,1}(A),\dots,\delta'_{f,N_f}(A)\bigr).
\end{align*}
So every target query of \(\mathcal P\) is reconstructed from finitely many source queries of \(\mathcal R\). Finally, for every \(A\in\Pi\),
\begin{align*}
	\Theta(A)
	&= D_{\mathcal{RQ}}(\Psi(E_{\mathcal{RQ}}(A)))\\
	&= D_{\mathcal{RQ}}(D_{\mathcal{QP}}(\Xi(E_{\mathcal{QP}}(E_{\mathcal{RQ}}(A)))))\\
	&= D_{\mathcal{RP}}(\Xi(E_{\mathcal{RP}}(A))).
\end{align*}
Hence all clauses of \cref{def:decoder-regular-fq-transport} are satisfied, and therefore
\[
\mathcal R \leq^{\mathscr R}_{\mathrm{fq}} \mathcal P.
\]
In other words, \(\leq^{\mathscr R}_{\mathrm{fq}}\) is transitive.
\end{proof}

\begin{remark}[The two basic instances]\label{rem:decoder-regular-fq-instances}
For \(\mathscr R=\mathscr R_{\mathrm{cont}}\), the relation $\leq^{\mathscr R_{\mathrm{cont}}}_{\mathrm{fq}}$ is exactly the finite-query evaluation reduction \(\leq_{G,\mathrm{fq}}\) from \cref{def:fq-eval-reduction}.

The Borel-decoder instance $\leq^{\mathscr R_{\mathrm{Bor}}}_{\mathrm{fq}}$ is included because decision/threshold quotients naturally lead to decoders that are Borel measurable without being continuous.
\end{remark}

\begin{definition}[Transport saturation]\label{def:transport-saturation}
Fix \(k\in\mathbb N\), an admissible decoder class \(\mathscr R\), a nonempty family $\mathcal F=(\mathcal P_i)_{i\in I}$, and a class \(\mathcal B\) of SCI computational problems. We say that \(\mathcal F\) is \textit{\((k,\mathcal B,\mathscr R)\)-transport-saturated} if
\[
\forall i\in I\ \exists \mathcal O_i\in \mathcal B\quad \mathcal O_i \leq^{\mathscr R}_{\mathrm{fq}} \mathcal P_i.
\]
If, in addition,
\[
\forall \mathcal O\in \mathcal B\quad \SCIG(\mathcal O)=k,
\]
we call \(\mathcal B\) an \textit{exact level-\(k\) transport basis} for \(\mathcal F\).
\end{definition}

\begin{corollary}[Transport saturation suffices for family-pointwise exactness]\label{cor:transport-saturation-suffices}
Fix \(k\in\mathbb N\). Let $\mathcal F=(\mathcal P_i)_{i\in I}$ be a nonempty family of SCI computational problems. Assume that \(\mathcal B\) is an exact level-\(k\) transport basis for \(\mathcal F\) with respect to \(\mathscr R_{\mathrm{cont}}\), and assume that \(\mathsf{UB}_k(\mathcal F)\) holds. Then \(\mathcal F\) is family-pointwise exact at height \(k\).
\end{corollary}

\begin{proof}
Fix \(i\in I\). By transport saturation there exists \(\mathcal O_i\in \mathcal B\) such that
\[
\mathcal O_i \leq^{\mathscr R_{\mathrm{cont}}}_{\mathrm{fq}} \mathcal P_i.
\]
By \cref{rem:decoder-regular-fq-instances}, this is exactly
\[
\mathcal O_i \leq_{G,\mathrm{fq}} \mathcal P_i.
\]
Hence \cref{thm:SCI-monotonicity-fq-reduction} yields
\[
\SCIG (\mathcal O_i)\le \SCIG (\mathcal P_i).
\]
Because \(\mathcal B\) is an exact level-\(k\) transport basis,
\[
\SCIG (\mathcal O_i)=k.
\]
Therefore
\[
k\le \SCIG (\mathcal P_i).
\]
Since \(\mathsf{UB}_k(\mathcal F)\) holds, we also have
\[
\SCIG (\mathcal P_i)\le k.
\]
Hence
\[
\SCIG (\mathcal P_i)=k.
\]
Since \(i\in I\) was arbitrary, \(\mathcal F\) is family-pointwise exact at height \(k\).
\end{proof}

\begin{remark}[Principal-source case]\label{rem:principal-source-special-case}
\Cref{cor:transport-saturation-suffices} reduces to \cref{cor:concrete-sufficiency-package} when \(\mathcal B=\{\mathcal S_k\}\) is a singleton. Thus \cref{cor:concrete-sufficiency-package} is precisely the principal case of transport saturation.
\end{remark}

\Cref{cor:transport-saturation-suffices} is still stated for a chosen family \(\mathcal F\). For the converse direction, however, it is useful to work inside a larger ambient class \(\mathcal U\) and ask how the exact height-\(k\) members of \(\mathcal U\) are organized by transport. The next definitions separate the principal case, where one source generates the whole exact layer, from the non-principal case, where several incompatible exact obstructions may be needed.

\begin{definition}[Exact level-\(k\) layer, exact transport basis, and principal upper cone]\label{def:exact-layer-basis-principal-upper-cone}
Let \(\mathcal U\) be a class of SCI computational problems and let \(k\in\mathbb N\). Define the exact level-\(k\) layer of \(\mathcal U\) by
\[
\mathcal O_k(\mathcal U):=\{\mathcal P\in \mathcal U:\SCIG (\mathcal P)=k\}.
\]

A class \(\mathcal B_k(U)\subseteq \mathcal U\) is called an \textit{exact level-\(k\) transport basis over \(\mathcal U\)} if
\[
\forall \mathcal O\in \mathcal B_k(U),\qquad \SCIG (\mathcal O)=k,
\]
and
\[
\mathcal O_k(\mathcal U) = \{\mathcal P\in \mathcal U:\exists \mathcal O\in \mathcal B_k(\mathcal U)\ \bigl(\mathcal O\le_{G,\mathrm{fq}} \mathcal P\bigr)\}.
\]

If \(\mathcal S_k\in \mathcal U\), define its upper cone inside \(\mathcal U\) by
\[
\uparrow_{\le_{G,\mathrm{fq}}}^{\,\mathcal U}(\mathcal S_k) := \{\mathcal P\in \mathcal U: \mathcal S_k\le_{G,\mathrm{fq}} \mathcal P\}.
\]
We say that \(\mathcal O_k(\mathcal U)\) is \textit{principal} if there exists \(\mathcal S_k\in \mathcal U\) such that
\[
\SCIG (\mathcal S_k)=k \quad\text{and}\quad
\mathcal O_k(\mathcal U)=\uparrow_{\le_{G,\mathrm{fq}}}^{\,\mathcal U}(\mathcal S_k).
\]
\end{definition}

\begin{proposition}[Principal ambient criterion]\label{prop:principal-ambient-criterion}
Let \(\mathcal U\) be a class of SCI computational problems and let \(k\in\mathbb N\). Assume
\[
\forall \mathcal P\in \mathcal U,\quad \SCIG (\mathcal P)\le k,
\]
and assume that there exists \(\mathcal S_k\in \mathcal U\) such that
\[
\SCIG (\mathcal S_k)=k \quad\text{and}\quad
\mathcal O_k(\mathcal U)=\uparrow_{\le_{G,\mathrm{fq}}}^{\,\mathcal U}(\mathcal S_k).
\]
Then for every nonempty family \(\mathcal F\subseteq \mathcal U\),
\[
\mathcal F \text{ is family-pointwise exact at height }k \iff
\forall \mathcal P\in \mathcal F,\ \mathcal S_k\le_{G,\mathrm{fq}} \mathcal P,
\]
and
\[
\mathcal F \text{ is witness-space sharp at height }k \iff
\exists \mathcal P\in \mathcal F,\ \mathcal S_k\le_{G,\mathrm{fq}} \mathcal P.
\]
\end{proposition}

\begin{proof}
Let \(\mathcal F\subseteq \mathcal U\) be nonempty. We first prove the criterion for family-pointwise exactness.
By definition,
\[
\mathcal F \text{ is family-pointwise exact at height }k \iff
\forall \mathcal P\in \mathcal F,\ \SCIG (\mathcal P)=k.
\]
Since
\[
\mathcal O_k(\mathcal U)=\uparrow_{\le_{G,\mathrm{fq}}}^{\,\mathcal U}(\mathcal S_k),
\]
this is equivalent to
\[
\forall \mathcal P\in \mathcal F,\ \mathcal P\in \uparrow_{\le_{G,\mathrm{fq}}}^{\,\mathcal U}(\mathcal S_k),
\]
i.e.
\[
\forall \mathcal P\in \mathcal F,\ \mathcal S_k\le_{G,\mathrm{fq}} \mathcal P.
\]

We now prove the witness-space criterion. Because
\[
\forall \mathcal P\in \mathcal U,\quad \SCIG (\mathcal P)\le k,
\]
a nonempty family \(\mathcal F\subseteq \mathcal U\) is witness-space sharp at height \(k\) iff
\[
\exists \mathcal P\in \mathcal F,\ \SCIG (\mathcal P)=k.
\]
Using again
\[
\mathcal O_k(\mathcal U)=\uparrow_{\le_{G,\mathrm{fq}}}^{\,\mathcal U}(\mathcal S_k),
\]
this is equivalent to
\[
\exists \mathcal P\in \mathcal F,\ \mathcal S_k\le_{G,\mathrm{fq}} \mathcal P.
\]
\end{proof}

\begin{remark}[Only the converse is nontrivial in the principal case]\label{rem:only-converse-nontrivial-principal-case}
Under the ambient upper-bound hypothesis
\[
\forall \mathcal P\in \mathcal U,\quad \SCIG (\mathcal P)\le k
\]
and the source exactness
\[
\SCIG (\mathcal S_k)=k,
\]
the implication
\[
\mathcal S_k\le_{G,\mathrm{fq}} \mathcal P \Longrightarrow \SCIG (\mathcal P)=k \quad (\mathcal P\in \mathcal U)
\]
is automatic by \cref{thm:SCI-monotonicity-fq-reduction}. Therefore the nontrivial part of the principal case is the converse implication
\[
\SCIG (\mathcal P)=k \Longrightarrow \mathcal S_k\le_{G,\mathrm{fq}} \mathcal P \quad (\mathcal P\in \mathcal U).
\]
Once this ambient-level principal converse is proved, the family-level principal exactness statements follow formally from \cref{prop:principal-ambient-criterion}.
\end{remark}

\begin{definition}[Pairwise exact-\(k\) transport incompatibility]\label{def:pairwise-exact-k-transport-incompatibility}
Let \(\mathcal U\) be a class of SCI computational problems and let \(k\in\mathbb N\). Two problems \(\mathcal O_0, \mathcal O_1\in \mathcal U\) are called \textit{pairwise exact-\(k\) transport-incompatible in \(\mathcal U\)} if
\[
\SCIG (\mathcal O_0)= \SCIG (\mathcal O_1)=k,
\]
and there does not exist any \(\mathcal S\in \mathcal U\) such that
\[
\SCIG (\mathcal S)=k, \quad \mathcal S\le_{G,\mathrm{fq}} \mathcal O_0, \quad
\mathcal S\le_{G,\mathrm{fq}} \mathcal O_1.
\]

A family \((\mathcal O_i)_{i\in I}\subseteq \mathcal U\) is called \textit{pairwise exact-\(k\) transport-incompatible} if every \(\mathcal O_i\) has exact type-\(G\) height \(k\) and for all distinct \(i,j\in I\), the pair \((\mathcal O_i, \mathcal O_j)\) is pairwise exact-\(k\) transport-incompatible in \(\mathcal U\).
\end{definition}

\begin{theorem}[Raw type-\(G\) non-principality criteria]\label{thm:raw-nonprincipality-criteria}
Let \(\mathcal U\) be a class of SCI computational problems and let \(k\in\mathbb N\). Then the following hold.

\begin{enumerate}
	\item If \(\mathcal U\) admits an exact level-\(k\) transport basis \(\mathcal B_k(\mathcal U)\) containing two pairwise exact-\(k\) transport-incompatible members, then \(\mathcal O_k(\mathcal U)\) is not principal.
	
	\item If \(\mathcal O_k(\mathcal U)\) contains an infinite pairwise exact-\(k\) transport-incompatible family, then \(\mathcal U\) admits no finite exact level-\(k\) transport basis.
\end{enumerate}
\end{theorem}

\begin{proof}
\noindent\textbf{(1):} Let \(\mathcal B_k(\mathcal U)\) be an exact level-\(k\) transport basis over \(\mathcal U\), and suppose that \(\mathcal O_0, \mathcal O_1\in \mathcal B_k(\mathcal U)\) are pairwise exact-\(k\) transport-incompatible. Assume, toward a contradiction, that \(\mathcal O_k(\mathcal U)\) is principal. Then there exists \(\mathcal S_k\in \mathcal U\) such that
\[
\SCIG (\mathcal S_k)=k \quad\text{and}\quad \mathcal O_k(\mathcal U)=\uparrow_{\le_{G,\mathrm{fq}}}^{\,\mathcal U}(\mathcal S_k).
\]
Since \(\mathcal B_k(\mathcal U)\) is an exact level-\(k\) transport basis, every element of \(\mathcal B_k(\mathcal U)\) has exact height \(k\). In particular,
\[
\mathcal O_0, \mathcal O_1\in \mathcal O_k(\mathcal U).
\]
Hence
\[
\mathcal S_k\le_{G,\mathrm{fq}} \mathcal O_0 \quad\text{and}\quad \mathcal S_k\le_{G,\mathrm{fq}} \mathcal O_1.
\]
Together with
\[
\SCIG (\mathcal S_k)=k,
\]
this contradicts the pairwise exact-\(k\) transport incompatibility of \(\mathcal O_0\) and \(\mathcal O_1\). Therefore \(\mathcal O_k(\mathcal U)\) is not principal.

\smallskip
\noindent\textbf{(2):} Assume that \((\mathcal O_n)_{n\in\mathbb N}\subseteq \mathcal \mathcal O_k(\mathcal U)\) is an infinite pairwise exact-\(k\) transport-incompatible family.
Assume, toward a contradiction, that \(\mathcal U\) admits a finite exact level-\(k\) transport basis
\[
\mathcal B_k(\mathcal U)=\{\mathcal S_1,\dots,\mathcal S_m\}.
\]
Because \(\mathcal B_k(\mathcal U)\) is an exact level-\(k\) transport basis, for every \(n\in\mathbb N\) there exists \(j(n)\in\{1,\dots,m\}\) such that
\[
\mathcal S_{j(n)}\le_{G,\mathrm{fq}} \mathcal O_n.
\]
By the pigeonhole principle, there exist distinct \(n\neq n'\) such that
\[
j(n)=j(n')=:j^\ast.
\]
Hence
\[
\mathcal S_{j^\ast}\le_{G,\mathrm{fq}} \mathcal O_n \quad\text{and}\quad
\mathcal S_{j^\ast}\le_{G,\mathrm{fq}} \mathcal O_{n'}.
\]
Since \(\mathcal B_k(\mathcal U)\) is an exact level-\(k\) transport basis, we also have
\[
\SCIG (\mathcal S_{j^\ast})=k.
\]
This contradicts the pairwise exact-\(k\) transport incompatibility of \(\mathcal O_n\) and \(\mathcal O_{n'}\). Therefore \(\mathcal U\) admits no finite exact level-\(k\) transport basis.
\end{proof}

\begin{corollary}[Transport saturation is stronger than witness-space sharpness]\label{cor:strictness-family-not-transport-saturated}
The strictness family
\[
        \mathcal F_{\mathrm{str}}=(\mathcal P_0, \mathcal P_{\exists})
\]
from \cref{def:strictness-family} is witness-space sharp at height \(1\), but it admits no exact level-\(1\) transport basis with respect to \(\mathcal R_{\mathrm{cont}}\).
\end{corollary}

\begin{proof}
By \cref{cor:source-family-strictness}, the family \(\mathcal F_{\mathrm{str}}\) is witness-space sharp at height \(1\), but is not family-pointwise exact at height \(1\). In particular,
\[
        \mathrm{SCI}_G(\mathcal P)\le 1
\]
for $\mathcal P\in\mathcal F_{\mathrm{str}}$, because
\[
        \mathrm{SCI}_G(\mathcal P_0)=0, \, \mathrm{SCI}_G(\mathcal P_{\exists})=1.
\]
Suppose, toward a contradiction, that \(\mathcal F_{\mathrm{str}}\) admitted an exact level-\(1\) transport basis with respect to \(\mathcal R_{\mathrm{cont}}\). Then \cref{cor:transport-saturation-suffices} would imply that \(\mathcal F_{\mathrm{str}}\) is family-pointwise exact at height \(1\). This contradicts \cref{cor:source-family-strictness}. Therefore no such exact level-\(1\) continuous transport basis exists.
\end{proof}

This negative corollary shows that continuous transport saturation is strictly stronger than witness-space sharpness. The next proposition shows, on a natural finite ambient class from spectral SCI, that non-principality can also occur because two exact problems are transport-incompatible.

\begin{proposition}[A natural non-principal ambient class at height \(3\)]\label{prop:natural-nonprincipal-ambient-height-three}
Adopt the blind discrete-spectrum exact-input SCI problems from \cite[\S 3.4.3]{ColbrookHansen22}:
\[
\mathcal P_{d,1}^{\mathrm{blind}} := \bigl( \Xi_1^d,\Omega_1^d,(\mathcal M_H,d_H),\Lambda_1^d \bigr),
\]
\[
\mathcal P_{d,2}^{\mathrm{blind}} := \bigl(\Xi_2^d,\Omega_2^d,(\{0,1\},d_{\mathrm{disc}}),\Lambda_2^d\bigr),
\]
where
\[
\Xi_1^d(A):= \overline{\sigma_d(A)}, \quad
\Xi_2^d(A):=\mathbf 1_{\{\sigma_d(A)\neq\varnothing\}},
\]
and
\[
d_{\mathrm{disc}}(u,v):=
\begin{cases}
	0,&u=v,\\
	1,&u\neq v.
\end{cases}
\]

Set
\[
\mathcal U_d^{\mathrm{np}}:=\{\mathcal P_{d,1}^{\mathrm{blind}}, \mathcal P_{d,2}^{\mathrm{blind}}\}.
\]
Then
\[
\mathcal O_3(\mathcal U_d^{\mathrm{np}})= \mathcal U_d^{\mathrm{np}},
\]
and the exact level-\(3\) layer of \(\mathcal U_d^{\mathrm{np}}\) is not principal.
\end{proposition}

\begin{proof}
By \cite[Thm.~3.17]{ColbrookHansen22}, one has
\[
\{\Xi_1^d,\Omega_1^d\}\in \Sigma^A_3\setminus \Delta^G_3, \quad
\{\Xi_2^d,\Omega_2^d\}\in \Sigma^A_3\setminus \Delta^G_3.
\]
Using the hierarchy relation \cite[(2.1)]{ColbrookHansen22} together with \cite[App.~A, Def.~A.10]{ColbrookHansen22}, it follows that
\[
\SCIG (\mathcal P_{d,1}^{\mathrm{blind}}) = \SCIG (\mathcal P_{d,2}^{\mathrm{blind}}) = 3.
\]
Hence
\[
\mathcal O_3(\mathcal U_d^{\mathrm{np}})= \mathcal U_d^{\mathrm{np}}.
\]

We now prove that the two members are pairwise exact-\(3\) transport-incompatible. First we show
\[
\mathcal P_{d,1}^{\mathrm{blind}}\not\le_{G,\mathrm{fq}} \mathcal P_{d,2}^{\mathrm{blind}}.
\]
Assume, toward a contradiction, that
\[
\mathcal P_{d,1}^{\mathrm{blind}}\le_{G,\mathrm{fq}} \mathcal P_{d,2}^{\mathrm{blind}}.
\]
Then \cref{def:fq-eval-reduction} provides a continuous decoder
\[
D:(\{0,1\},d_{\mathrm{disc}})\to (\mathcal M_H,d_H).
\]
Hence the image of \(D\) has cardinality at most \(2\), so the target map \(\Xi_1^d\) could take at most two values on \(\Omega_1^d\).

For \(\lambda\in\{1,2,3\}\), define the diagonal operator
\[
D_\lambda e_1:=\lambda e_1, \quad
D_\lambda e_n:=0 \quad (n\ge 2).
\]
These are bounded normal diagonal operators and belong to \(\Omega_1^d\). By the definition of the discrete spectrum for normal operators in \cite[§3.4]{ColbrookHansen22}, one has
\[
\sigma_d(D_\lambda)=\{\lambda\},
\]
and therefore
\[
\Xi_1^d(D_\lambda)=\{\lambda\}.
\]
So \(\Xi_1^d\) takes at least three distinct values on \(\Omega_1^d\), a contradiction. Thus
\[
\mathcal P_{d,1}^{\mathrm{blind}}\not\le_{G,\mathrm{fq}} \mathcal P_{d,2}^{\mathrm{blind}}.
\]

Next we show
\[
\mathcal P_{d,2}^{\mathrm{blind}}\not\le_{G,\mathrm{fq}} \mathcal P_{d,1}^{\mathrm{blind}}.
\]
Assume, toward a contradiction, that
\[
\mathcal P_{d,2}^{\mathrm{blind}}\le_{G,\mathrm{fq}} \mathcal P_{d,1}^{\mathrm{blind}}.
\]
Then \cref{def:fq-eval-reduction} provides a continuous decoder
\[
D:(\mathcal M_H,d_H)\to (\{0,1\},d_{\mathrm{disc}}).
\]
We claim that \(D\) must be constant. Indeed, \(\mathcal M_H\) is path connected: given \(K,L\in \mathcal M_H\), define
\[
\gamma:[0,1]\to \mathcal M_H
\]
by
\[
\gamma(t):=
\begin{cases}
	(1-2t)K,&0\le t\le \frac12,\\[0.3em]
	(2t-1)L,&\frac12\le t\le 1.
\end{cases}
\]
At \(t=\frac12\), both branches equal \(\{0\}\), so \(\gamma\) is continuous in the Hausdorff metric and joins \(K\) to \(L\). Since the codomain \((\{0,1\},d_{\mathrm{disc}})\) is discrete, every continuous map
\[
(\mathcal M_H,d_H)\to (\{0,1\},d_{\mathrm{disc}})
\]
is constant. Therefore \(\Xi_2^d\) would be constant on \(\Omega_2^d\), a contradiction. Indeed, the operator \(D_1\) above satisfies
\[
\sigma_d(D_1)=\{1\},
\]
so
\[
\Xi_2^d(D_1)=1.
\]
On the other hand, if \((q_n)_{n\in\mathbb N}\subseteq [0,1]\) has dense range and
\[
Ne_n:=q_n e_n,
\]
then \(N\) is bounded normal diagonal, belongs to \(\Omega_2^d\), and has empty discrete spectrum, so
\[
\Xi_2^d(N)=0.
\]
Thus \(\Xi_2^d\) is not constant. Hence
\[
\mathcal P_{d,2}^{\mathrm{blind}}\not\le_{G,\mathrm{fq}} \mathcal P_{d,1}^{\mathrm{blind}}.
\]
Therefore the two members of \(\mathcal U_d^{\mathrm{np}}\) are pairwise exact-\(3\) transport-incompatible. Since \(\mathcal U_d^{\mathrm{np}}\) itself is an exact level-\(3\) transport basis over itself (by reflexivity of \(\le_{G,\mathrm{fq}}\)), \cref{thm:raw-nonprincipality-criteria}(1) applies and yields that the exact level-\(3\) layer $\mathcal O_3(\mathcal U_d^{\mathrm{np}}$) is not principal.
\end{proof}

\begin{definition}[Decoder-regular finite-query transport degrees]\label{def:transport-degrees}
Fix an admissible decoder class \(\mathscr R\). For SCI computational problems \(\mathcal P,\mathcal Q\), define
\[
\mathcal P \equiv^{\mathscr R}_{\mathrm{fq}} \mathcal Q :\Longleftrightarrow
\mathcal P \leq^{\mathscr R}_{\mathrm{fq}} \mathcal Q \ \wedge\ \mathcal Q \leq^{\mathscr R}_{\mathrm{fq}} \mathcal P.
\]
The quotient classes under \(\equiv^{\mathscr R}_{\mathrm{fq}}\) are called the \(\mathscr R\)-\textit{finite-query transport degrees}.
\end{definition}

By \cref{prop:decoder-regular-fq-preorder}, the quotient by \(\equiv^{\mathscr R}_{\mathrm{fq}}\) is a partially ordered degree structure. A natural question from computability theory is now, if it has a lattice structure. The answer is no in full generality as we will prove now.

\begin{proposition}[The transport-degree quotient need not be a lattice]\label{prop:transport-degrees-not-lattice}
There exists an admissible decoder class \(\mathscr R_{\mathrm{id}}\) such that the partially ordered set of \(\mathscr R_{\mathrm{id}}\)-finite-query transport degrees is not even an upper
semilattice. In particular, it is not a lattice.
\end{proposition}

\begin{proof}
Define
\[
\mathscr R_{\mathrm{id}} := \bigl\{ \mathrm{id}_{(\mathcal M,d)}:(\mathcal M,d)\to (\mathcal M,d) :\ (\mathcal M,d)\ \text{a metric space} \bigr\}.
\]
We first check that \(\mathscr R_{\mathrm{id}}\) is an admissible decoder class in the sense of \cref{def:admissible-decoder-class}. (1) is obviously fulfilled, so all we have to show is the closure under composition. If
\[
f:(\mathcal M,d)\to(\mathcal M,d),\quad g:(\mathcal N,\rho)\to(\mathcal N,\rho)
\]
belong to \(\mathscr R_{\mathrm{id}}\) and the composition \(g\circ f\) is defined, then necessarily $(\mathcal M,d)=(\mathcal N,\rho)$ as domain/codomain and
\[
f=g=\mathrm{id}_{(\mathcal M,d)}.
\]
Hence
\[
g\circ f=\mathrm{id}_{(\mathcal M,d)}\in \mathscr R_{\mathrm{id}}.
\]
Therefore \(\mathscr R_{\mathrm{id}}\) is admissible.

\smallskip
We now construct two SCI computational problems whose
\(\mathscr R_{\mathrm{id}}\)-transport degrees have no common upper bound.

Let
\[
\Omega_0:=\{\ast_0\},\quad \Omega_1:=\{\ast_1\}.
\]
Let the output spaces be the metric spaces
\[
\mathcal M_0:=\{0\},\quad \mathcal M_1:=\{1\},
\]
equipped with the unique metrics
\[
d_0(0,0):=0,\quad d_1(1,1):=0.
\]
Define target maps
\[
\Xi_0:\Omega_0\to \mathcal M_0,\qquad \Xi_0(\ast_0):=0,
\]
and
\[
\Xi_1:\Omega_1\to \mathcal M_1,\qquad \Xi_1(\ast_1):=1.
\]
Define evaluation families
\[
\Lambda_0:=\{c_0\},\qquad c_0(\ast_0):=0,
\]
and
\[
\Lambda_1:=\{c_1\},\qquad c_1(\ast_1):=0.
\]
Set
\[
\mathcal P_0:=(\Xi_0,\Omega_0,(\mathcal M_0,d_0),\Lambda_0), \quad
\mathcal P_1:=(\Xi_1,\Omega_1,(\mathcal M_1,d_1),\Lambda_1).
\]

These are SCI computational problems: the consistency condition from \cref{def:SCIcompProb} holds vacuously because each input class is a singleton.

We claim that there is no SCI computational problem \(\mathcal R\) such that
\[
\mathcal P_0 \leq^{\mathscr R_{\mathrm{id}}}_{\mathrm{fq}} \mathcal R \quad\text{and}\quad
\mathcal P_1 \leq^{\mathscr R_{\mathrm{id}}}_{\mathrm{fq}} \mathcal R.
\]
Suppose, toward a contradiction, that such an \(\mathcal R\) exists. Write
\[
\mathcal R=(\Xi_{\mathcal R},\Omega_{\mathcal R},(\mathcal M_{\mathcal R},d_{\mathcal R}),\Lambda_{\mathcal R}).
\]

From
\[
\mathcal P_0 \leq^{\mathscr R_{\mathrm{id}}}_{\mathrm{fq}} \mathcal R
\]
we obtain, by \cref{def:decoder-regular-fq-transport}, a decoder
\[
D_0:(\mathcal M_{\mathcal R},d_{\mathcal R})\to (\mathcal M_0,d_0)
\]
with
\[
D_0\in \mathscr R_{\mathrm{id}}.
\]
But every element of \(\mathscr R_{\mathrm{id}}\) is an identity map on its own domain/codomain. Hence \(D_0\) can exist only if $(\mathcal M_{\mathcal R},d_{\mathcal R})=(\mathcal M_0,d_0)$ and $D_0=\mathrm{id}_{(\mathcal M_0,d_0)}$. Therefore $\mathcal M_{\mathcal R}=\mathcal M_0=\{0\}$.

Likewise, from
\[
\mathcal P_1 \leq^{\mathscr R_{\mathrm{id}}}_{\mathrm{fq}} \mathcal R
\]
we obtain a decoder
\[
D_1:(\mathcal M_{\mathcal R},d_{\mathcal R})\to (\mathcal M_1,d_1)
\]
with
\[
D_1\in \mathscr R_{\mathrm{id}}.
\]
Again this is possible only if $(\mathcal M_{\mathcal R},d_{\mathcal R})=(\mathcal M_1,d_1)$, hence $\mathcal M_{\mathcal R}=\mathcal M_1=\{1\}$.

Combining gives $\{0\}=\mathcal M_{\mathcal R}=\{1\}$, which is impossible.
This contradiction proves that \(\mathcal P_0\) and \(\mathcal P_1\) have no common upper bound with respect to \(\leq^{\mathscr R_{\mathrm{id}}}_{\mathrm{fq}}\).

Now pass to the quotient by
\[
\equiv^{\mathscr R_{\mathrm{id}}}_{\mathrm{fq}}.
\]
Let
\[
[\mathcal P_0],\ [\mathcal P_1]
\]
denote the corresponding transport degrees. If the quotient poset were an upper semilattice, then \([\mathcal P_0]\) and \([\mathcal P_1]\) would admit a least upper bound. In particular, they would admit
some common upper bound. Since the quotient order is induced by the preorder \(\leq^{\mathscr R_{\mathrm{id}}}_{\mathrm{fq}}\), this would imply the existence of an SCI computational problem \(\mathcal R\) with
\[
\mathcal P_0 \leq^{\mathscr R_{\mathrm{id}}}_{\mathrm{fq}} \mathcal R \quad\text{and}\quad
\mathcal P_1 \leq^{\mathscr R_{\mathrm{id}}}_{\mathrm{fq}} \mathcal R,
\]
contradicting what we just proved. Therefore the degree structure for \(\mathscr R_{\mathrm{id}}\) is not an upper semilattice, and hence not a lattice.
\end{proof}

\begin{remark}[What the appendix proves for the natural decoder classes]\label{rem:transport-degrees-open-for-cont-bor}
\Cref{prop:transport-degrees-not-lattice} refutes lattice structure in full generality for arbitrary admissible decoder classes. \Cref{sec:LattTranspQuot} strengthens this for the two natural decoder classes
$\mathscr R_{\mathrm{cont}}$ and $\mathscr R_{\mathrm{Bor}}$ on the full class of SCI computational problems the corresponding quotients are not upper semilattices, hence not lattices. On the nondegenerate subclass
\[
\mathcal C_{\neq\varnothing,\neq\varnothing} :=
\{\mathcal P=(\Xi,\Omega,(M,d),\Lambda):\Omega\neq\varnothing,\ \Lambda\neq\varnothing\},
\]
the preorder is upward directed, and on
\[
\mathcal C_{\neq\varnothing} :=
\{ \mathcal P=(\Xi,\Omega,(\mathcal M,d),\Lambda):\Omega\neq\varnothing\},
\]
it is downward directed. Thus the remaining natural-class question is whether binary least upper bounds and greatest lower bounds exist on the nondegenerate subclass.
\end{remark}

\section{Natural Families Where The Upgrade Theorem Applies}\label{sec:interinteupgrade}
We now record natural infinite families of exact point-evaluation problems to which the principal case from \cref{cor:transport-saturation-suffices} applies; even stronger: these are families to which the concrete sufficiency package from \cref{cor:concrete-sufficiency-package} applies.
The first subsection treats interval integration. The second gives a spectral decision problem from the SCI literature whose family-pointwise exactness is stable under adjoining a fixed spectrally irrelevant block.

\subsection{Exact Integration On Compact Intervals}{\label{sec:ExactIntegrationCpt}}
The family considered in this section is the classical univariate numerical integration problem with function values as information.
For background on numerical quadrature see \cite[Ch.~1-2]{DavisRabinowitz84}; for the information-based complexity viewpoint on continuous problems with
finitely many function values see \cite[Ch.~3-4]{TraubWasilkowskiWozniakowski88} and for the broader standard-information integration perspective see \cite[Ch.~9]{NovakWozniakowski10Vol2}.

\begin{definition}[Exact point-evaluation integration on compact intervals]\label{def:interval-integration}
Let
\[
\mathcal I:=\{[a,b]\subseteq \mathbb R : a<b\}.
\]
For \(I=[a,b]\in\mathcal I\), define
\[
\Omega_I:=C(I,\mathbb R),
\qquad
\Lambda_I:=\{\operatorname{ev}_x:x\in I\},
\qquad
\operatorname{ev}_x(f):=f(x)\in \mathbb R\subseteq \mathbb C.
\]
Since every continuous function on a compact interval is Riemann integrable \cite[Thm.~6.8]{Rudin76}, the map
\[
\Xi_I:\Omega_I\to \mathbb R, \qquad
\Xi_I(f):=\int_a^b f(x)\,dx
\]
is well defined.

If \(f,g\in \Omega_I\) satisfy \(\Xi_I(f)\neq \Xi_I(g)\), then \(f\neq g\).
Hence there exists \(x\in I\) with \(f(x)\neq g(x)\), i.e.
\(\operatorname{ev}_x(f)\neq \operatorname{ev}_x(g)\).
Therefore
\[
\mathcal P_I^{\mathrm{int}} :=
\bigl(\Xi_I,\Omega_I,(\mathbb R,|\cdot|),\Lambda_I\bigr)
\]
is a computational problem in the sense of \cref{def:SCIcompProb}.

Define the family
\[
\mathcal F^{\mathrm{int}} :=
\bigl(\mathcal P_I^{\mathrm{int}}\bigr)_{I\in\mathcal I}.
\]
\end{definition}

\begin{lemma}[Uniform one-limit upper bound for interval integration]\label{lem:interval-integration-upper}
For every \(I=[a,b]\in\mathcal I\),
\[
\SCIG \bigl(\mathcal P_I^{\mathrm{int}}\bigr)\le 1.
\]
\end{lemma}
\begin{proof}
Fix \(I=[a,b]\in\mathcal I\). For every \(n\in\mathbb N\) and \(j\in\{0,\ldots,n\}\), set
\[
x^{I,n}_j:=a+j\frac{b-a}{n}.
\]
Define
\[
\Gamma_n^I:\Omega_I\to \mathbb R, \qquad
\Gamma_n^I(f):=\frac{b-a}{n}\sum_{j=0}^{n-1} f \left(x^{I,n}_j\right),
\]
and define the query-set map
\[
\Lambda_{\Gamma_n^I}(f) :=
\left\{\operatorname{ev}_{x^{I,n}_j}:0\le j\le n-1 \right\} \qquad (f\in\Omega_I).
\]
This queried set is finite and nonempty, so
\[
\Lambda_{\Gamma_n^I}(f)\in [\Lambda_I]^{<\omega}\setminus\{\emptyset\}.
\]

We first check that \((\Gamma_n^I,\Lambda_{\Gamma_n^I})\) is a general algorithm on
\(\mathcal P_I^{\mathrm{int}}\). Let \(f,g\in \Omega_I\) and assume that
\[
\eta(g)=\eta(f)\qquad\text{for all }\eta\in \Lambda_{\Gamma_n^I}(f).
\]
Since \(\Lambda_{\Gamma_n^I}(f)\) is independent of \(f\), one has
\[
\Lambda_{\Gamma_n^I}(g)=\Lambda_{\Gamma_n^I}(f).
\]
Moreover,
\[
g \left(x^{I,n}_j \right)=f \left(x^{I,n}_j \right) \qquad (j=0,\ldots,n-1),
\]
hence
\[
\Gamma_n^I(g) =
\frac{b-a}{n}\sum_{j=0}^{n-1} g \left(x^{I,n}_j \right) =
\frac{b-a}{n}\sum_{j=0}^{n-1} f \left(x^{I,n}_j \right) = \Gamma_n^I(f).
\]
Thus \((\Gamma_n^I,\Lambda_{\Gamma_n^I})\) is a general algorithm.

It remains to prove convergence. Fix \(f\in\Omega_I\) and \(\varepsilon>0\).
Because \(f\) is continuous on the compact interval \(I\), it is uniformly continuous \cite[Thm.~4.19]{Rudin76}.
Hence there exists \(\delta>0\) such that
\[
|x-y|<\delta \Longrightarrow |f(x)-f(y)|<\frac{\varepsilon}{b-a} \qquad (x,y\in I).
\]
Choose \(n\in\mathbb N\) so large that
\[
\frac{b-a}{n}<\delta.
\]
Then for every \(j\in\{0,\ldots,n-1\}\) and every \(x\in[x^{I,n}_j,x^{I,n}_{j+1}]\), one has
\[
|x-x^{I,n}_j|\le \frac{b-a}{n}<\delta,
\]
hence
\[
|f(x)-f(x^{I,n}_j)|<\frac{\varepsilon}{b-a}.
\]
Therefore
\begin{align*}
	\left| \Xi_I(f)-\Gamma_n^I(f) \right| &=
	\left| \sum_{j=0}^{n-1} \left( \int_{x^{I,n}_j}^{x^{I,n}_{j+1}} f(x)\,dx - \frac{b-a}{n}f(x^{I,n}_j) \right) \right| \\
	&\le \sum_{j=0}^{n-1} \int_{x^{I,n}_j}^{x^{I,n}_{j+1}} |f(x)-f(x^{I,n}_j)|\,dx \\
	&<\sum_{j=0}^{n-1} \frac{b-a}{n}\cdot \frac{\varepsilon}{b-a} = \varepsilon.
\end{align*}
Thus
\[
\lim_{n\to\infty}\Gamma_n^I(f)=\Xi_I(f) \qquad (f\in\Omega_I).
\]
Hence \(\bigl(\Gamma_n^I \bigr)_{n\in\mathbb N}\) is a height-\(1\) type-\(G\) tower for \(\mathcal P_I^{\mathrm{int}}\), proving
\[
\SCIG \bigl(\mathcal P_I^{\mathrm{int}} \bigr)\le 1.
\qedhere
\]
\end{proof}

\begin{remark}
The tower in \cref{lem:interval-integration-upper} is the classical composite left-endpoint rectangle rule;
compare the standard quadrature discussion in \cite[Ch.~2, \S2.4]{DavisRabinowitz84}.
\end{remark}

\begin{proposition}[The unit-interval source has exact height \(1\)]\label{prop:interval-integration-source}
Let
\[
\mathcal S_1:=\mathcal P_{[0,1]}^{\mathrm{int}}.
\]
Then $\SCIG(\mathcal S_1)=1.$
\end{proposition}
\begin{proof}
It remains to prove that \(\mathrm{SCI}_G(S_1)\neq 0\) by \cref{lem:interval-integration-upper}.
Assume, toward a contradiction, that \(\SCIG(\mathcal S_1)=0\).
Then by \cref{def:GenAlgoSCI} there exists a general algorithm $(\Gamma,\Lambda_\Gamma)$ on \(\mathcal S_1\) such that $\Gamma=\Xi_{[0,1]}.$

Let
\[
f_0:[0,1]\to\mathbb R, \qquad f_0(x):=0.
\]
Write
\[
\Lambda_\Gamma(f_0) = \{\operatorname{ev}_{x_1},\ldots,\operatorname{ev}_{x_m}\}
\]
with \(x_1,\ldots,x_m\in[0,1]\). Since the set \(\{x_1,\ldots,x_m\}\) is finite, there exist numbers \(u,v\) with
\[
0<u<v<1, \qquad [u,v]\cap\{x_1,\ldots,x_m\}=\emptyset.
\]
Define
\[
h:[0,1]\to\mathbb R, \qquad
h(x):= \max \left\{ 1-\frac{2}{v-u}\left|x-\frac{u+v}{2}\right|, \,0 \right\}.
\]
Then \(h\in C([0,1],\mathbb R)\), and since \(h(x)=0\) for every \(x\notin [u,v]\), one has
\[
h(x_j)=0=f_0(x_j) \qquad (j=1,\ldots,m).
\]
Hence
\[
\eta(h)=\eta(f_0) \qquad\text{for all }\eta\in\Lambda_\Gamma(f_0).
\]
Because \((\Gamma,\Lambda_\Gamma)\) is a general algorithm, it follows that
\[
\Gamma(h)=\Gamma(f_0).
\]
Using \(\Gamma=\Xi_{[0,1]}\), we obtain
\[
\int_0^1 h(x)\,dx =
\int_0^1 f_0(x)\,dx = 0.
\]
On the other hand,
\begin{align*}
	\int_0^1 h(x)\,dx
	&= \int_u^v h(x)\,dx \\
	&= 2\int_0^{(v-u)/2} \left(1-\frac{2t}{v-u}\right)\,dt \\
	&= 2\left[ t-\frac{t^2}{v-u} \right]_{0}^{(v-u)/2} \\
	&= 2\left( \frac{v-u}{2}-\frac{v-u}{4} \right) \\
	&= \frac{v-u}{2} > 0,
\end{align*}
a contradiction.

Therefore \(\SCIG(\mathcal S_1)\ge 1\).
\end{proof}

The bump-function obstruction in \cref{prop:interval-integration-source} is the exact-computation analogue of the basic indistinguishability phenomenon in information-based
complexity: finitely many function values do not determine the value of the integral on the full class $C([0,1],\mathbb R)$; compare the general framework in \cite[Ch.~3-4]{TraubWasilkowskiWozniakowski88}.

\begin{proposition}[Family-wide finite-query evaluation reduction from the unit interval]\label{prop:interval-integration-reduction}
For every \(I=[a,b]\in\mathcal I\),
\[
\mathcal S_1\le_{G,\mathrm{fq}} \mathcal P_I^{\mathrm{int}}.
\]
\end{proposition}

\begin{proof}
Fix \(I=[a,b]\in\mathcal I\). Define the encoding map
\[
E_I:\Omega_{[0,1]}\to \Omega_I, \qquad (E_I f)(x):=\frac{1}{b-a}\, f \left(\frac{x-a}{b-a} \right).
\]
Since \(f\in C([0,1],\mathbb R)\) and the affine map \(x\mapsto \frac{x-a}{b-a}\) is continuous from \(I\) to \([0,1]\),
one has \(E_I f\in C(I,\mathbb R)=\Omega_I\).

Now define the decoding map
\[
D_I:\mathbb R\to\mathbb R, \qquad D_I(y):=y.
\]

Now let \(f\in\Omega_{[0,1]}\). By the change-of-variable theorem for the Riemann-Stieltjes integral \cite[Thm.~6.19]{Rudin76} and using the affine substitution
\[
t=\frac{x-a}{b-a}, \qquad x=a+(b-a)t
\]
we have
\begin{align*}
	D_I \bigl(\Xi_I(E_I f) \bigr)
	&= \int_a^b \frac{1}{b-a}\, f \left(\frac{x-a}{b-a} \right)\,dx \\
	&= \int_0^1 f(t)\,dt \\
	&= \Xi_{[0,1]}(f).
\end{align*}
Thus the target relation required in \cref{def:fq-eval-reduction} holds.

It remains to verify the finite-query simulation of evaluations. Fix \(x\in I\), and consider the evaluation
\[
\operatorname{ev}_x\in \Lambda_I.
\]
Set
\[
m_{\operatorname{ev}_x}:=1, \qquad \gamma_{\operatorname{ev}_x,1} := \operatorname{ev}_{(x-a)/(b-a)} \in \Lambda_{[0,1]}.
\]
Define further
\[
\vartheta_{\operatorname{ev}_x}:
\operatorname{im}(\gamma_{\operatorname{ev}_x,1})\to \mathbb C, \qquad \vartheta_{\operatorname{ev}_x}(z):=\frac{1}{b-a}z.
\]
Then for every \(f\in\Omega_{[0,1]}\),
\begin{align*}
	\operatorname{ev}_x(E_I f)
	&= (E_I f)(x) \\
	&= \frac{1}{b-a}\, f \left(\frac{x-a}{b-a} \right) \\
	&= \vartheta_{\operatorname{ev}_x} \Bigl( \gamma_{\operatorname{ev}_x,1}(f) \Bigr).
\end{align*}
Hence each evaluation in \(\Lambda_I\) is simulated from exactly one source evaluation in \(\Lambda_{[0,1]}\).

All clauses of \cref{def:fq-eval-reduction} are therefore satisfied, and so
\[
\mathcal S_1\le_{G,\mathrm{fq}} \mathcal P_I^{\mathrm{int}}.
\qedhere
\]
\end{proof}

\begin{theorem}[Family-pointwise exactness at height \(1\) for interval integration]\label{thm:interval-integration-family-exact}
The family \(\mathcal F^{\mathrm{int}}\) is family-pointwise exact at height \(1\). Equivalently,
\[
\forall I\in\mathcal I, \qquad
\SCIG \bigl(\mathcal P_I^{\mathrm{int}} \bigr)=1.
\]
\end{theorem}

\begin{proof}
Let
\[
\mathcal S_1=\mathcal P_{[0,1]}^{\mathrm{int}}.
\]
By \cref{prop:interval-integration-source} $\SCIG(\mathcal S_1)=1$ and by \cref{lem:interval-integration-upper}, for every \(I\in\mathcal I\) $\SCIG \bigl(\mathcal P_I^{\mathrm{int}} \bigr)\le 1.$
Thus \(\mathsf{UB}_1(\mathcal F^{\mathrm{int}})\) holds.

Moreover, \cref{prop:interval-integration-reduction} shows that for every \(I\in\mathcal I\),
\[
\mathcal S_1\le_{G,\mathrm{fq}} \mathcal P_I^{\mathrm{int}}.
\]
Hence the conditions \((C1)\) - \((C3)\) of \cref{cor:concrete-sufficiency-package} are satisfied with
\[
k=1, \qquad
\mathcal F= \mathcal F^{\mathrm{int}}, \qquad \mathcal S_1= \mathcal P_{[0,1]}^{\mathrm{int}}.
\]
Therefore \cref{cor:concrete-sufficiency-package} yields family-pointwise exactness at height \(1\).
\end{proof}

The family \(\mathcal F^{\mathrm{int}}\) already gives a direct positive application of \cref{cor:concrete-sufficiency-package}. To obtain an ambient converse in the sense of \cref{sec:decoder-regular-transport}, it is natural to enlarge this family by adjoining the degenerate intervals \([a,a]\). These form a canonical lower-height layer and thereby isolate the exact level-\(1\) layer inside a single ambient class.

\begin{definition}[The interval ambient class with degenerate members]\label{def:degenerate-interval-ambient}
Define
\[
\overline{\mathcal I}:=\{[a,b]\subseteq \mathbb R:a\le b\},
\qquad
\mathcal I_+:=\{[a,b]\subseteq \mathbb R:a<b\},
\qquad
\mathcal I_0:=\{[a,a]:a\in\mathbb R\}.
\]
For each \(I=[a,b]\in\overline{\mathcal I}\), let
\[
\mathcal P_I^{\mathrm{int}} :=
\bigl(\Xi_I,\Omega_I,(\mathbb R,|\cdot|),\Lambda_I \bigr),
\]
where \(\Omega_I,\Lambda_I,\Xi_I\) are the exact point-evaluation integration data from \cref{def:interval-integration}. Define
\[
\mathcal U_{\mathrm{int}}:=\{ \mathcal P_I^{\mathrm{int}}:I\in\overline{\mathcal I}\}, \quad \mathcal S_1:= \mathcal P_{[0,1]}^{\mathrm{int}}.
\]
\end{definition}

\begin{lemma}[Degenerate intervals have exact height \(0\)]\label{lem:degenerate-interval-height-zero}
For every \(a\in \mathbb R\),
\[
\SCIG \bigl( \mathcal P_{[a,a]}^{\mathrm{int}} \bigr)=0.
\]
\end{lemma}

\begin{proof}
Let $I:=[a,a]$. Then
\[
\Omega_I=C(\{a\},\mathbb R), \quad \Lambda_I=\{\mathrm{ev}_a\}, \quad \Xi_I(f)=\int_a^a f(x)\,dx=0 \quad (f\in \Omega_I).
\]
So the target map is constant. Define
\[
\Gamma_I:\Omega_I\to \mathbb R, \quad \Gamma_I(f):=0,
\]
and define the finite-query map
\[
\Lambda_{\Gamma_I}(f):=\{\mathrm{ev}_a\} \quad (f\in \Omega_I).
\]
Since \(\Gamma_I\) is constant and \(\Lambda_{\Gamma_I}\) is constant, the defining condition for a general algorithm from \cref{def:GenAlgoSCI}(i) is immediate: if
\[
\eta(g)=\eta(f) \quad
\forall \eta\in \Lambda_{\Gamma_I}(f),
\]
then certainly
\[
\Gamma_I(g)=0=\Gamma_I(f) \quad\text{and}\quad \Lambda_{\Gamma_I}(g)=\{\mathrm{ev}_a\}=\Lambda_{\Gamma_I}(f).
\]
Thus \((\Gamma_I,\Lambda_{\Gamma_I})\) is a general algorithm on \(\mathcal P_I^{\mathrm{int}}\). Since also $\Gamma_I=\Xi_I$, this is a height-\(0\) type-\(G\) tower computing \(\Xi_I\). Therefore
\[
\SCIG \bigl(\mathcal P_{[a,a]}^{\mathrm{int}} \bigr)=0.
\]
\end{proof}

\begin{theorem}[A principal ambient converse on the interval family]\label{thm:interval-principal-converse}
For every \(I=[a,b]\in\overline{\mathcal I}\),
\[
\SCIG \bigl( \mathcal P_I^{\mathrm{int}} \bigr)=1 \iff
\mathcal S_1\le_{G,\mathrm{fq}} \mathcal P_I^{\mathrm{int}}.
\]
Equivalently,
\[
\mathcal \mathcal O_1(\mathcal U_{\mathrm{int}}) = \uparrow_{\le_{G,\mathrm{fq}}}^{\,\mathcal U_{\mathrm{int}}}(\mathcal S_1) = \{ \mathcal P_I^{\mathrm{int}}: I\in \mathcal I_+\}.
\]
\end{theorem}

\begin{proof}
We prove the two implications.

\smallskip \noindent\textbf{(\(\Rightarrow\)):} Assume
\[
\SCIG \bigl( \mathcal P_I^{\mathrm{int}} \bigr)=1.
\]
Then \(I\notin \mathcal I_0\), because \cref{lem:degenerate-interval-height-zero} gives exact height \(0\) for degenerate intervals. Hence \(I=[a,b]\in \mathcal I_+\), and \cref{prop:interval-integration-reduction} gives
\[
\mathcal S_1\le_{G,\mathrm{fq}} \mathcal P_I^{\mathrm{int}}.
\]

\smallskip \noindent\textbf{(\(\Leftarrow\)):}
Assume
\[
\mathcal S_1\le_{G,\mathrm{fq}} \mathcal P_I^{\mathrm{int}}.
\]
By \cref{thm:SCI-monotonicity-fq-reduction},
\[
\SCIG (\mathcal S_1)\le \SCIG \bigl( \mathcal P_I^{\mathrm{int}} \bigr).
\]
By \cref{prop:interval-integration-source} $\SCIG (\mathcal S_1)=1$, so
\[
1\le \SCIG \bigl( \mathcal P_I^{\mathrm{int}} \bigr).
\]
If \(I\in \mathcal I_+\), then \cref{lem:interval-integration-upper} gives
\[
\SCIG \bigl( \mathcal P_I^{\mathrm{int}} \bigr)\le 1.
\]
If \(I\in \mathcal I_0\), then \cref{lem:degenerate-interval-height-zero} gives
\[
\SCIG \bigl( \mathcal P_I^{\mathrm{int}} \bigr)=0,
\]
which is incompatible with \(1\le \SCIG (\mathcal P_I^{\mathrm{int}})\). So the case \(I\in \mathcal I_0\) is impossible. Therefore necessarily \(I\in \mathcal I_+\), and in that case
\[
\SCIG \bigl(\mathcal P_I^{\mathrm{int}} \bigr)\le 1.
\]
Combining with the lower bound yields
\[
\SCIG \bigl( \mathcal P_I^{\mathrm{int}} \bigr)=1.
\]
The set-theoretic reformulation is exactly the same equivalence written inside the ambient class \(\mathcal U_{\mathrm{int}}\).
\end{proof}

\begin{corollary}[Family criterion on the interval ambient class]\label{cor:interval-ambient-family-criterion}
For every nonempty family \(\mathcal F\subseteq \mathcal U_{\mathrm{int}}\),
\[
\mathcal F \text{ is family-pointwise exact at height } 1 \iff
\forall \mathcal P\in \mathcal F,\ \mathcal S_1\le_{G,\mathrm{fq}} \mathcal P.
\]
Moreover,
\[
\mathcal F \text{ is witness-space sharp at height }1 \iff
\exists \mathcal P\in \mathcal F,\ \mathcal S_1\le_{G,\mathrm{fq}} \mathcal P.
\]
\end{corollary}

\begin{proof}
The equivalence
\[
\mathcal \mathcal O_1(\mathcal U_{\mathrm{int}}) = \uparrow_{\le_{G,\mathrm{fq}}}^{\,\mathcal U_{\mathrm{int}}}(\mathcal S_1)
\]
is exactly \cref{thm:interval-principal-converse}. Since every member of \(\mathcal U_{\mathrm{int}}\) has SCI \(\le 1\) by \cref{lem:interval-integration-upper} and \cref{lem:degenerate-interval-height-zero}, \cref{prop:principal-ambient-criterion} applies with \(\mathcal U= \mathcal U_{\mathrm{int}}\) and \(k=1\), and yields the two equivalences.
\end{proof}

\subsection{A Spectral-Decision Family From Block-Diagonal Stabilization}\label{sec:block-diagonal-stabilization}

We next give a natural example coming directly from the spectral SCI literature. To keep the source classification completely rigorous at exactly the point needed here, we work with singleton windows inside a fixed compact real interval and with the diagonal operator subclass from \cite[\S3.2.1]{ColbrookHansen22}.
This is the fixed-window version covered directly by \cite[Thm.~3.10 and Rem.~3.11]{ColbrookHansen22}.

Fix a nonempty compact interval $J\subset \mathbb R$. We will use in this subsection the standard spectral notation $\rho(T), \sigma(T)$ for the resolvent set and spectrum of an (possibly unbounded) operator $T$. Let $H:=\ell^2(\mathbb N)$ with canonical basis $(e_n)_{n\in\mathbb N}$, and let
\[
d_{\mathrm{disc}}(u,v):=
\begin{cases}
0,&u=v,\\
1,&u\neq v
\end{cases}
\qquad (u,v\in\{0,1\}).
\]

Let \(\mathcal M_H\) denote the hyperspace of nonempty compact subsets of \(\mathbb C\), equipped with the Hausdorff metric \(d_H\). We write
\[
\mathcal K_{\mathrm{sgl}}(J):=\bigl\{\{z\}:z\in J\bigr\}
\subseteq \mathcal M_H
\]
for the class of singleton compact windows contained in $J$.

For $K=\{z\}\in \mathcal K_{\mathrm{sgl}}(J)$ and $n\in\mathbb N$, define
\[
r_n(K):=2^{-(n+2)}\Bigl\lfloor 2^{n+2}z\Bigr\rfloor \in \mathbb Q,
\qquad
K_n(K):=\{r_n(K)\}.
\]
Then
\[
d_H\!\bigl(K_n(K),K\bigr)=|r_n(K)-z|<2^{-(n+2)}.
\]

Let
\[
\Omega_{\mathrm{diag}}:=\Omega_D
\]
be the diagonal operator class from \cite[\S3.2.1]{ColbrookHansen22}.
For $(A,K)\in \Omega_{\mathrm{diag}}\times \mathcal K_{\mathrm{sgl}}(J)$, define
\[
\mu_{i,j}(A,K):=\langle A e_j,e_i\rangle
\qquad (i,j\in\mathbb N),
\]
and
\[
\rho_n(A,K):=r_n(K)
\qquad (n\in\mathbb N).
\]
Set
\[
\Lambda_J^{\mathrm{sgl}}
:=
\{\mu_{i,j}:i,j\in\mathbb N\}\cup\{\rho_n:n\in\mathbb N\}.
\]

Define the target map
\[
\Xi_J^{\mathrm{sgl}}:
\Omega_{\mathrm{diag}}\times \mathcal K_{\mathrm{sgl}}(J)\to \{0,1\},
\qquad
\Xi_J^{\mathrm{sgl}}(A,K):=
\mathbf 1_{\{\sigma(A)\cap K=\varnothing \}}.
\]
We therefore obtain the computational problem
\[
\mathcal S_J^{\mathrm{sgl}} :=
\Bigl(
\Xi_J^{\mathrm{sgl}},
\Omega_{\mathrm{diag}}\times \mathcal K_{\mathrm{sgl}}(J),
(\{0,1\},d_{\mathrm{disc}}),
\Lambda_J^{\mathrm{sgl}}
\Bigr).
\]

\begin{remark}
The consistency requirement from \cref{def:SCIcompProb} holds for
$\mathcal S_J^{\mathrm{sgl}}$.
Indeed, the matrix elements $\mu_{i,j}$ determine the diagonal operator
$A\in\Omega_{\mathrm{diag}}$, and the sequence $(\rho_n(K))_{n\in\mathbb N}$
determines the singleton $K=\{z\}$ because $\rho_n(K)\to z$.
\end{remark}

We use in the next proposition the hierarchy notation \(\Delta^G_m\), \(\Delta^A_m\), and \(\Pi^A_m\) from \cite[§2.1 and App.~A, Def.~A.10]{ColbrookHansen22}. In particular, membership in \(\Delta^G_{m+1}\) means solvability by a type-G tower of height \(m\).

\begin{proposition}[Exact type-$G$ height of the singleton-window source problem]\label{prop:singleton-window-source}
One has
\[
\SCIG \bigl(\mathcal S_J^{\mathrm{sgl}} \bigr)=2.
\]
\end{proposition}

\begin{proof}
By \cite[\S3.2.1]{ColbrookHansen22}, the class $\Omega_{\mathrm{diag}}$ is the diagonal
subclass of the graph/operator framework to which \cite[Thm.~3.10]{ColbrookHansen22} applies.
For diagonal operators, the bounded-dispersion data from \cite[(3.9)]{ColbrookHansen22} are universal: one may take
\[
f(n)=n,
\qquad
c_n=0
\qquad (n\in\mathbb N),
\]
because all off-diagonal matrix entries vanish identically.
Moreover, our evaluations $\rho_n$ provide exactly the compact-window approximants required in \cite[Thm.~3.10]{ColbrookHansen22}, and the restriction to
singleton windows is justified by \cite[Rem.~3.11]{ColbrookHansen22}, since
\[
d_H \bigl(K_n(K),K\bigr)<2^{-(n+2)}<2^{-(n+1)}.
\]

Hence \cite[Thm.~3.10 and Rem.~3.11]{ColbrookHansen22} apply directly to the present
singleton-window problem on the diagonal class, and yield
\[
\mathcal S_J^{\mathrm{sgl}}\notin \Delta_2^G
\qquad\text{and}\qquad
\mathcal S_J^{\mathrm{sgl}}\in \Pi_2^A.
\]
By the schematic hierarchy relation in \cite[(2.1)]{ColbrookHansen22},
\[
\Pi_2^A\subseteq \Delta_3^A.
\]
Every type-$A$ tower is, by definition, a particular type-$G$ tower, hence
\[
\mathcal S_J^{\mathrm{sgl}}\in \Delta_3^G.
\]
By the SCI-hierarchy definitions \cite[\S2.1 and App.~A, Def.~A.10]{ColbrookHansen22},
\[
\SCIG \bigl(\mathcal S_J^{\mathrm{sgl}} \bigr)=2.
\qedhere
\]
\end{proof}

Now let
\[
\mathcal B_J:=
\bigl\{B\in \Omega_{\mathrm{diag}}:\sigma(B)\cap J=\varnothing\bigr\}.
\]
This class is nonempty; for example, if $\lambda\in\mathbb R\setminus J$, then
$\lambda I\in \mathcal B_J$.

Fix $B\in \mathcal B_J$.
Let
\[
H^{(2)}:=H\oplus H,
\]
and write
\[
u_n^{(1)}:=(e_n,0), \quad
u_n^{(2)}:=(0,e_n) \quad (n\in\mathbb N)
\]
for the canonical basis of $H^{(2)}$ adapted to the direct sum decomposition.

Define the input class
\[
\Omega_{J,B}^{\mathrm{sgl}} :=
\bigl\{(A\oplus B,K): A\in \Omega_{\mathrm{diag}}, \ K\in \mathcal K_{\mathrm{sgl}}(J) \bigr\}.
\]
For $(T,K)\in \Omega_{J,B}^{\mathrm{sgl}}$ define
\[
\nu_{(i,r),(j,s)}^B(T,K) :=
\langle T u_j^{(s)},u_i^{(r)}\rangle \quad (i,j\in\mathbb N,\ r,s\in\{1,2\}),
\]
and
\[
\rho_n^B(T,K):=r_n(K) \quad (n\in\mathbb N).
\]
Set
\[
\Lambda_{J,B}^{\mathrm{sgl}} :=
\bigl\{ \nu_{(i,r),(j,s)}^B : i,j\in\mathbb N,\ r,s\in\{1,2\} \bigr\} \cup \{\rho_n^B:n\in\mathbb N\}.
\]
Define
\[
\Xi_{J,B}^{\mathrm{sgl}}: \Omega_{J,B}^{\mathrm{sgl}}\to \{0,1\},\quad
\Xi_{J,B}^{\mathrm{sgl}}(T,K) := \mathbf 1_{\{\sigma(T)\cap K=\varnothing\}}.
\]
Thus
\[
\mathcal P_{J,B}^{\mathrm{sgl}} :=
\Bigl( \Xi_{J,B}^{\mathrm{sgl}}, \Omega_{J,B}^{\mathrm{sgl}}, (\{0,1\},d_{\mathrm{disc}}), \Lambda_{J,B}^{\mathrm{sgl}} \Bigr)
\]
is a computational problem.

\begin{remark}
The consistency assumption from \cref{def:SCIcompProb} also holds for
$\mathcal P_{J,B}^{\mathrm{sgl}}$.
Indeed, the matrix elements $\nu_{(i,r),(j,s)}^B$ determine the operator $T=A\oplus B$, and the sequence $(\rho_n^B(T,K))_{n\in\mathbb N}$ determines the singleton $K$.
\end{remark}

\begin{proposition}[Two-sided finite-query transport under block-diagonal stabilization]\label{prop:block-diagonal-two-sided}
For every $B\in\mathcal B_J$, one has
\[
\mathcal S_J^{\mathrm{sgl}}\le_{G,\mathrm{fq}} \mathcal P_{J,B}^{\mathrm{sgl}} \quad\text{and}\quad
\mathcal P_{J,B}^{\mathrm{sgl}}\le_{G,\mathrm{fq}} \mathcal S_J^{\mathrm{sgl}}.
\]
\end{proposition}

\begin{proof}
We first prove
\[
\mathcal S_J^{\mathrm{sgl}}\le_{G,\mathrm{fq}} \mathcal P_{J,B}^{\mathrm{sgl}}.
\]

Define the encoding map
\[
E_B: \Omega_{\mathrm{diag}}\times \mathcal K_{\mathrm{sgl}}(J) \to \Omega_{J,B}^{\mathrm{sgl}}, \quad E_B(A,K):=(A\oplus B,K).
\]
Define the decoding map
\[
D_B:\{0,1\}\to\{0,1\}, \quad D_B(y):=y.
\]
This decoder is continuous because the codomain is discrete.

We now verify the target relation. For any closed operators $C$ and $D$ and any $\lambda\in\mathbb C$,
\[
(C\oplus D)-\lambda I=(C-\lambda I)\oplus(D-\lambda I)
\]
by \cite[Ch.~III, \S5.6, (5.23)-(5.24)]{Kato95}. Hence $\lambda\in\rho(C\oplus D)$ if and only if $\lambda\in\rho(C)\cap\rho(D)$:
if both resolvents exist, then
\[
\bigl((C-\lambda I)\oplus(D-\lambda I)\bigr)^{-1} =
(C-\lambda I)^{-1}\oplus(D-\lambda I)^{-1},
\]
and conversely, if the direct sum is invertible, its inverse preserves the two
coordinate subspaces and therefore restricts to inverses of $C-\lambda I$ and
$D-\lambda I$ (see also \cite[Prob.~III.5.37]{Kato95}). Thus
\[
\sigma(C\oplus D)=\sigma(C)\cup \sigma(D).
\]
Applying this with $C=A$ and $D=B$, and using
\[
K\subseteq J, \quad \sigma(B)\cap J=\varnothing,
\]
we get
\[
\sigma(A\oplus B)\cap K=\varnothing \iff \sigma(A)\cap K=\varnothing.
\]
Therefore
\[
D_B \bigl(\Xi_{J,B}^{\mathrm{sgl}}(E_B(A,K))\bigr) =\Xi_{J,B}^{\mathrm{sgl}}(A\oplus B,K)= \Xi_J^{\mathrm{sgl}}(A,K).
\]

It remains to verify finite-query simulation. Fix the dummy source evaluation
\[
\delta:=\rho_1\in\Lambda_J^{\mathrm{sgl}}.
\]

For a window query $\rho_n^B\in\Lambda_{J,B}^{\mathrm{sgl}}$, set
\[
m_{\rho_n^B}:=1, \quad \gamma_{\rho_n^B,1}:=\rho_n, \quad \vartheta_{\rho_n^B}(z):=z.
\]
Then
\[
\rho_n^B(E_B(A,K))=r_n(K)=\rho_n(A,K) =\vartheta_{\rho_n^B}\bigl(\gamma_{\rho_n^B,1}(A,K)\bigr).
\]

For a first-block operator query
\[
\nu_{(i,1),(j,1)}^B\in\Lambda_{J,B}^{\mathrm{sgl}},
\]
set
\[
m_{\nu_{(i,1),(j,1)}^B}:=1,\quad \gamma_{\nu_{(i,1),(j,1)}^B,1}:=\mu_{i,j}, \quad \vartheta_{\nu_{(i,1),(j,1)}^B}(z):=z.
\]
Then
\[
\nu_{(i,1),(j,1)}^B(E_B(A,K)) = \langle A e_j,e_i\rangle = \mu_{i,j}(A,K) = \vartheta_{\nu_{(i,1),(j,1)}^B} \bigl(\gamma_{\nu_{(i,1),(j,1)}^B,1}(A,K)\bigr).
\]

For a second-block operator query
\[
\nu_{(i,2),(j,2)}^B,
\]
set
\[
m_{\nu_{(i,2),(j,2)}^B}:=1, \quad
\gamma_{\nu_{(i,2),(j,2)}^B,1}:=\delta, \quad \vartheta_{\nu_{(i,2),(j,2)}^B}(z):=\langle B e_j,e_i\rangle.
\]
Since $B$ is fixed,
\[
\nu_{(i,2),(j,2)}^B(E_B(A,K)) = \langle B e_j,e_i\rangle =
\vartheta_{\nu_{(i,2),(j,2)}^B} \bigl(\gamma_{\nu_{(i,2),(j,2)}^B,1}(A,K)\bigr).
\]

For a mixed-block query
\[
\nu_{(i,1),(j,2)}^B \quad\text{or}\quad \nu_{(i,2),(j,1)}^B,
\]
set
\[
m_f:=1, \quad \gamma_{f,1}:=\delta, \quad \vartheta_f(z):=0.
\]
Because $A\oplus B$ is block diagonal,
\[
f(E_B(A,K))=0=\vartheta_f(\gamma_{f,1}(A,K)).
\]

Hence all clauses of \cref{def:fq-eval-reduction} are satisfied, proving
\[
\mathcal S_J^{\mathrm{sgl}}\le_{G,\mathrm{fq}} \mathcal P_{J,B}^{\mathrm{sgl}}.
\]

We now prove
\[
\mathcal P_{J,B}^{\mathrm{sgl}}\le_{G,\mathrm{fq}} \mathcal S_J^{\mathrm{sgl}}.
\]
Given $(T,K)\in \Omega_{J,B}^{\mathrm{sgl}}$, there exists a unique$A\in\Omega_{\mathrm{diag}}$ such that
\[
T=A\oplus B.
\]
Define
\[
F_B: \Omega_{J,B}^{\mathrm{sgl}} \to \Omega_{\mathrm{diag}}\times \mathcal K_{\mathrm{sgl}}(J), \quad
F_B(T,K):=(A,K),
\]
where $A$ is this unique first block. Let
\[
E'_B:=F_B, \quad D'_B:=\mathrm{id}_{\{0,1\}}.
\]
The same spectral-union argument as above yields
\[
\Xi_{J,B}^{\mathrm{sgl}}(T,K) = \Xi_J^{\mathrm{sgl}}(E'_B(T,K)).
\]

For a source window query $\rho_n$, use the family query $\rho_n^B$:
\[
m_{\rho_n}:=1, \quad \gamma_{\rho_n,1}:=\rho_n^B, \quad
\vartheta_{\rho_n}(z):=z.
\]
Then
\[
\rho_n(E'_B(T,K))=r_n(K)=\rho_n^B(T,K).
\]

For a source operator query $\mu_{i,j}$, use the first-block family query
$\nu_{(i,1),(j,1)}^B$:
\[
m_{\mu_{i,j}}:=1,\quad
\gamma_{\mu_{i,j},1}:=\nu_{(i,1),(j,1)}^B,\quad \vartheta_{\mu_{i,j}}(z):=z.
\]
Then, if $T=A\oplus B$,
\[
\mu_{i,j}(E'_B(T,K))=\langle A e_j,e_i\rangle = \langle T u_j^{(1)},u_i^{(1)}\rangle = \nu_{(i,1),(j,1)}^B(T,K).
\]

Thus all clauses of \cref{def:fq-eval-reduction} hold, and therefore
\[
\mathcal P_{J,B}^{\mathrm{sgl}}\le_{G,\mathrm{fq}} \mathcal S_J^{\mathrm{sgl}}.
\qedhere
\]
\end{proof}

\begin{theorem}[Family-pointwise exactness for spectrally irrelevant block stabilization]\label{thm:block-diagonal-family-exact}
For every $B\in\mathcal B_J$,
\[
\SCIG \bigl(\mathcal P_{J,B}^{\mathrm{sgl}} \bigr)=2.
\]
Consequently, the family
\[
\mathcal F_{J,\mathrm{stab}}^{\mathrm{sgl}} :=
\bigl(\mathcal P_{J,B}^{\mathrm{sgl}} \bigr)_{B\in\mathcal B_J}
\]
is family-pointwise exact at height $2$.
In particular, it is witness-space sharp at height $2$ and worst-case exact at
height $2$.
\end{theorem}

\begin{proof}
Fix $B\in\mathcal B_J$. By \cref{prop:singleton-window-source},
\[
\SCIG \bigl(\mathcal S_J^{\mathrm{sgl}} \bigr)=2.
\]
By \cref{prop:block-diagonal-two-sided} and \cref{thm:SCI-monotonicity-fq-reduction},
\[
\SCIG \bigl(\mathcal S_J^{\mathrm{sgl}} \bigr) \le
\SCIG \bigl(\mathcal P_{J,B}^{\mathrm{sgl}} \bigr)\le \SCIG \bigl(\mathcal S_J^{\mathrm{sgl}} \bigr).
\]
Hence
\[
\SCIG \bigl(\mathcal P_{J,B}^{\mathrm{sgl}} \bigr)=2.
\]
Since $B\in\mathcal B_J$ was arbitrary, every family member has exact height $2$,
so $\mathcal F_{J,\mathrm{stab}}^{\mathrm{sgl}}$ is family-pointwise exact at height $2$.
The final sentence follows from \cref{prop:LogRelNot}.
\end{proof}

\begin{remark}[Why this is the literature-backed fixed-window version]
The block-diagonal stabilization mechanism itself works much more generally. We formulate it here for singleton windows contained in a fixed compact interval
\(J\) because this is the version whose exact source classification is provided directly by \cite[Thm.~3.10 and Rem.~3.11]{ColbrookHansen22}.
\end{remark}

\begin{corollary}[Rigid principal calibration of the stabilized spectral ambient]\label{cor:rigid-principal-calibration-stabilized-spectral-ambient}
Set
\[
\mathcal U_{J,\mathrm{stab}} := \{\mathcal P_{J,B}^{\mathrm{sgl}}:B\in \mathcal B_J\}.
\]
Then
\[
        \mathcal O_2(\mathcal U_{J,\mathrm{stab}})=\mathcal U_{J,\mathrm{stab}},
\]
and
\[
        \mathcal U_{J,\mathrm{stab}} = \{\mathcal P\in \mathcal U_{J,\mathrm{stab}} : \mathcal S_J^{\mathrm{sgl}}\le_{G,fq} \mathcal P\} =
        \{\mathcal P\in \mathcal U_{J,\mathrm{stab}} : \mathcal P\equiv_{G,fq} \mathcal S_J^{\mathrm{sgl}}\}.
\]
If one enlarges the ambient class to
\[
        \mathcal U^+_{J,\mathrm{stab}} := \mathcal U_{J,\mathrm{stab}} \cup \{\mathcal S_J^{\mathrm{sgl}}\},
\]
then \(\mathcal U^+_{J,\mathrm{stab}}\) is principal at height \(2\) in the sense of \cref{def:exact-layer-basis-principal-upper-cone}, with principal generator \(\mathcal S_J^{\mathrm{sgl}}\).
\end{corollary}

\begin{proof}
By \cref{thm:block-diagonal-family-exact},
\[
\forall \mathcal P\in \mathcal U_{J,\mathrm{stab}},\quad \SCIG (\mathcal P)=2.
\]
Hence
\[
\mathcal O_2(\mathcal U_{J,\mathrm{stab}})=\mathcal U_{J,\mathrm{stab}}.
\]

By \cref{prop:block-diagonal-two-sided}, for every \(B\in \mathcal B_J\),
\[
\mathcal S_J^{\mathrm{sgl}}\le_{G,\mathrm{fq}} \mathcal P_{J,B}^{\mathrm{sgl}} \quad\text{and}\quad
\mathcal P_{J,B}^{\mathrm{sgl}}\le_{G,\mathrm{fq}} \mathcal S_J^{\mathrm{sgl}}.
\]
After adjoining \(\mathcal S_J^{\mathrm{sgl}}\) itself, the principal-ambient statement follows directly from \cref{def:exact-layer-basis-principal-upper-cone}.
\end{proof}

\begin{remark}[Transport-preorder reading]
The family \(\mathcal F^{\mathrm{sgl}}_{J,\mathrm{stab}}\) is generated by the single external source \(\mathcal S_J^{\mathrm{sgl}}\) under continuous finite-query transport. \Cref{prop:block-diagonal-two-sided} gives the stronger two-sided relation
\[
        \mathcal S_J^{\mathrm{sgl}}\le_{G,fq} \mathcal P^{\mathrm{sgl}}_{J,B} \quad\text{and}\quad
        \mathcal P^{\mathrm{sgl}}_{J,B}\le_{G,fq} \mathcal S_J^{\mathrm{sgl}}
\]
for $B\in B_J$. Thus every stabilized member lies in the same continuous finite-query transport degree as the source. If the source is adjoined to the ambient class, this becomes a principalambient in the literal sense of \cref{def:exact-layer-basis-principal-upper-cone}.
\end{remark}

\section{Limitations And Open Problems}{\label{sec:openproblems}}

The present paper gives a raw type-\(G\) sufficient transport theory for family exactness, but the converse directions remain open.

A separate question raised by \cref{sec:decoder-regular-transport} concerns the transport-degree structures associated with the natural decoder classes \(\mathscr R_{\mathrm{cont}}\) and \(\mathscr R_{\mathrm{Bor}}\).
\Cref{sec:LattTranspQuot} settles the unrestricted case negatively: once empty evaluation families are allowed, the corresponding quotients are not upper semilattices, hence not lattices.
On the other hand, \cref{sec:LattTranspQuot} also shows that on the nondegenerate subclass with
\[
\Omega\neq\varnothing \quad\text{and}\quad
\Lambda\neq\varnothing,
\]
the preorder is upward directed, and on the subclass with \(\Omega\neq\varnothing\) it is downward directed. Thus, for the natural decoder classes, the remaining degree-theoretic question is whether
binary least upper bounds and binary greatest lower bounds exist on the nondegenerate subclass.

The main open problems are therefore

\textbf{Open problem 1:}
The results of \cref{sec:AbstrUpgrTh} - \cref{sec:interinteupgrade} show that the converse from witness-space sharpness to family-pointwise exactness should be formulated in terms of exact transport bases on suitable ambient classes.

Let \(\mathcal U\) be a natural class of SCI computational problems and let \(k\in\mathbb N\). Determine whether the exact level-\(k\) layer
\[
\mathcal O_k(\mathcal U)=\{ \mathcal P\in \mathcal U: \SCIG (\mathcal P)=k\}
\]
admits an exact level-\(k\) transport basis \(\mathcal B_k(\mathcal U)\subseteq \mathcal U\), i.e.
\[
\mathcal O_k(\mathcal U) = \{\mathcal P\in \mathcal U:\exists \mathcal O\in \mathcal B_k(\mathcal U)\ \bigl( \mathcal O\le_{G,\mathrm{fq}} \mathcal P \bigr)\}.
\]
Further, classify when such a basis is principal, finitely generated non-principal, or complete non-principal. Finally, construct natural ambient assignments \(\mathcal F\mapsto \mathcal U(\mathcal F)\) for broad classes of families \(\mathcal F\), so that the resulting ambient exactness criterion yields a family-level necessary-and-sufficient condition for the upgrade from witness-space sharpness to family-pointwise exactness.

\textbf{Open problem 2:}
Fix $R\in\{\mathscr R_{\mathrm{cont}}, \mathscr R_{\mathrm{Bor}}\}$. Let
\[
\mathcal C_{\neq\varnothing,\neq\varnothing} :=
\bigl\{ \mathcal P=(\Xi,\Omega,(\mathcal M,d),\Lambda): \Omega\neq\varnothing,\ \Lambda\neq\varnothing \bigr\}
\]
be the nondegenerate subclass of SCI computational problems.

\Cref{sec:LattTranspQuot} shows that on the full class of SCI computational problems the quotient by $\equiv^\mathscr R_{\mathrm{fq}}$ is not an upper semilattice, hence not a lattice, and that on \(\mathcal C_{\neq\varnothing,\neq\varnothing}\) the preorder $\leq^\mathscr R_{\mathrm{fq}}$ is upward directed, while on
\[
\mathcal C_{\neq\varnothing} :=
\bigl\{ \mathcal P=(\Xi,\Omega,(\mathcal M,d),\Lambda): \Omega\neq\varnothing \bigr\}
\]
it is downward directed.

Determine whether the quotient induced by \(\equiv^\mathscr R_{\mathrm{fq}}\) on \(\mathcal C_{\neq\varnothing,\neq\varnothing}\) is (i) an upper semilattice, (ii) a lower semilattice, (iii) a lattice.

\textbf{Acknowledgments } The introduced, analyzed and used transport-obstruction-theory in \cref{sec:decoder-regular-transport} was highly inspired by the work in \cite{NP19}, as the author first constructed and tried understanding a continuous-exact-obstruction-base-theory for the continuous intermediate case introduced in \cite{Sor26} using the language of parametrized Wadge games.

\textbf{Statement } During the preparation of this work the author used UniBwM-ChatGPT5.4 in order to improve the language in the abstract and introduction. After using this tool, the author reviewed and edited the content as needed and takes full responsibility for the content of the published article.

\textbf{Conflict of interest } The author declares no conflict of interest.

\appendix
\section{(Non)-Lattice Structure Of Transport-Degree Quotients: Natural Classes}{\label{sec:LattTranspQuot}}

In this appendix we analyze the two natural decoder classes $\mathscr R_{\mathrm{cont}}$ and $\mathscr R_{\mathrm{Bor}}$ from \cref{def:admissible-decoder-class}. We first show that on the \textit{full} class of SCI computational problems the corresponding transport-degree quotients are not upper semilattices.
We then show that on the nondegenerate subclass with nonempty input class and nonempty evaluation family, the easy obstruction disappears: every pair has a common upper bound, and in fact also a common lower bound.

\begin{remark}[Empty evaluation families are allowed]\label{rem:empty-evaluation-families}
\Cref{def:SCIcompProb} does not require the evaluation family \(\Lambda\) to be nonempty. In particular, if the target map \(\Xi\) is constant, then \(\Lambda=\varnothing\) is allowed,
since the consistency condition
\[
\Xi(A)\neq \Xi(B)\Longrightarrow \exists f\in\Lambda\ (f(A)\neq f(B))
\]
is then vacuous.
\end{remark}

This convention concerns SCI computational problems as objects of the transport preorder. It does not imply that every such problem has a height-zero algorithm under \cref{def:GenAlgoSCI}, because the present definition of a general algorithm requires a nonempty finite query set.

\begin{proposition}[The full \(\mathscr R_{\mathrm{cont}}\)- and \(\mathscr R_{\mathrm{Bor}}\)-degree structures are not upper semilattices]\label{prop:full-cont-bor-not-upper-semilattice}
For each
\[
R\in\{\mathscr R_{\mathrm{cont}}, \mathscr R_{\mathrm{Bor}}\},
\]
the quotient of SCI computational problems by \(\equiv^\mathscr R_{\mathrm{fq}}\) is not an upper semilattice. In particular, it is not a lattice.
\end{proposition}

\begin{proof}
Fix
\[
R\in\{\mathscr R_{\mathrm{cont}}, \mathscr R_{\mathrm{Bor}}\}.
\]
Define
\[
\Omega_0:=\{\ast\}, \quad \mathcal M_0:=\{0\},
\]
where \(\mathcal M_0\) carries the unique metric \(d_0\), and let
\[
\Xi_0:\Omega_0\to \mathcal M_0, \quad
\Xi_0(\ast):=0.
\]
Set
\[
\Lambda_0:=\varnothing.
\]
Then
\[
\mathcal P_0:=(\Xi_0,\Omega_0,(\mathcal M_0,d_0),\Lambda_0)
\]
is an SCI computational problem by \cref{rem:empty-evaluation-families}. Next define
\[
\Omega_1:=\{a,b\}, \quad
\mathcal M_1:=\{0,1\},
\]
where \(\mathcal M_1\) carries the discrete metric
\[
d_{\mathrm{disc}}(u,v):=
\begin{cases}
	0,&u=v,\\
	1,&u\neq v.
\end{cases}
\]
Let
\[
\Xi_1:\Omega_1\to \mathcal M_1, \quad \Xi_1(a):=0, \quad
\Xi_1(b):=1,
\]
and define
\[
e:\Omega_1\to\mathbb C, \quad e(a):=0, \quad e(b):=1.
\]
Set
\[
\Lambda_1:=\{e\}.
\]
Then
\[
\mathcal P_1:=(\Xi_1,\Omega_1,(\mathcal M_1,d_{\mathrm{disc}}),\Lambda_1)
\]
is an SCI computational problem, since \(\Xi_1(a)\neq \Xi_1(b)\) and the single evaluation \(e\) separates \(a\) and \(b\).

We claim that there is no SCI computational problem
\[
\mathcal U=(\Xi_{\mathcal U},\Omega_{\mathcal U},(\mathcal M_{\mathcal U},d_{\mathcal U}),\Lambda_{\mathcal U})
\]
such that
\[
\mathcal P_0\leq^\mathscr R_{\mathrm{fq}} \mathcal U \quad\text{and}\quad \mathcal P_1\leq^\mathscr R_{\mathrm{fq}} \mathcal U.
\]
Assume, toward a contradiction, that such a \(\mathcal U\) exists. From
\[
\mathcal P_0\leq^\mathscr R_{\mathrm{fq}} \mathcal U,
\]
\cref{def:decoder-regular-fq-transport} yields the following: for every \(f\in\Lambda_{\mathcal U}\), there must exist a natural number \(m_f\in\mathbb N\) and source
evaluations
\[
\gamma_{f,1},\dots,\gamma_{f,m_f}\in\Lambda_0.
\]
But \(\Lambda_0=\varnothing\). Since \(m_f\in\mathbb N\), this is impossible unless
\[
\Lambda_{\mathcal U}=\varnothing.
\]

We now show that \(\Xi_{\mathcal U}\) must be constant on \(\Omega_{\mathcal U}\).
Suppose not. Then there exist \(A,B\in\Omega_{\mathcal U}\) with
\[
\Xi_{\mathcal U}(A)\neq \Xi_{\mathcal U}(B).
\]
Since \(\mathcal U\) is an SCI computational problem, the consistency condition from \cref{def:SCIcompProb} implies that there exists \(f\in\Lambda_{\mathcal U}\) such that $f(A)\neq f(B)$. But \(\Lambda_{\mathcal U}=\varnothing\), a contradiction. Hence \(\Xi_{\mathcal U}\) is constant on \(\Omega_{\mathcal U}\).

Now use
\[
\mathcal P_1\leq^\mathscr R_{\mathrm{fq}} \mathcal U.
\]
By \cref{def:decoder-regular-fq-transport}, there exist an encoding
\[
E_1:\Omega_1\to\Omega_{\mathcal U}
\]
and a decoder
\[
D_1:(\mathcal M_{\mathcal U},d_{\mathcal U})\to(\mathcal M_1,d_{\mathrm{disc}})
\]
such that for every \(x\in\Omega_1\),
\[
\Xi_1(x)=D_1(\Xi_{\mathcal U}(E_1(x))).
\]
Since \(\Xi_{\mathcal U}\) is constant, the right-hand side is constant in \(x\). Hence \(\Xi_1\) must be constant on \(\Omega_1\), contradicting
\[
\Xi_1(a)\neq \Xi_1(b).
\]

This contradiction proves that \(\mathcal P_0\) and \(\mathcal P_1\) have no common upper bound with respect to
\[
\leq^\mathscr R_{\mathrm{fq}}.
\]
Passing to the quotient by \(\equiv^\mathscr R_{\mathrm{fq}}\), the corresponding degrees still have no common upper bound. Therefore the quotient poset is not an upper semilattice.
\end{proof}

The preceding counterexample exploits an empty evaluation family. To isolate the nondegenerate situation relevant for the natural decoder classes, we now restrict to problems with
nonempty input class and nonempty evaluation family.

\begin{definition}[The nondegenerate subclasses]\label{def:nondegenerate-subclasses}
Define
\[
\mathcal C_{\neq\varnothing} := \bigl\{ \mathcal P=(\Xi,\Omega,(\mathcal M,d),\Lambda):\Omega\neq\varnothing \bigr\},
\]
and
\[
\mathcal C_{\neq\varnothing,\neq\varnothing} :=
\bigl\{ \mathcal P=(\Xi,\Omega,(\mathcal M,d),\Lambda):\Omega\neq\varnothing,\ \Lambda\neq\varnothing \bigr\}.
\]
\end{definition}

\begin{proposition}[Upward directedness on the nonempty-query subclass]\label{prop:upward-directed-nonempty}
Let
\[
R\in\{\mathscr R_{\mathrm{cont}}, \ \mathscr R_{\mathrm{Bor}}\}.
\]
Let
\[
\mathcal P_i=(\Xi_i,\Omega_i,(\mathcal M_i,d_i),\Lambda_i),
\]
where $i=0,1$, be SCI computational problems belonging to \(\mathcal C_{\neq\varnothing,\neq\varnothing}\). Then there exists an SCI computational problem \(\mathcal U\) such that
\[
\mathcal P_0\leq^\mathscr R_{\mathrm{fq}} \mathcal U \quad\text{and}\quad \mathcal P_1\leq^\mathscr R_{\mathrm{fq}} \mathcal U.
\]
In particular, the preorder \(\leq^\mathscr R_{\mathrm{fq}}\) is upward directed on \(\mathcal C_{\neq\varnothing,\neq\varnothing}\).
\end{proposition}

\begin{proof}
Fix
\[
\mathcal P_i=(\Xi_i,\Omega_i,(\mathcal M_i,d_i),\Lambda_i),
\]
where $i=0,1$, with
\[
\Omega_i\neq\varnothing, \quad \Lambda_i\neq\varnothing.
\]
Choose
\[
A_i^\ast\in\Omega_i \quad\text{and}\quad q_i\in\Lambda_i,
\]
where $i=0,1$, which exists by the nonemptiness assumptions. Set
\[
m_i^\ast:=\Xi_i(A_i^\ast)\in \mathcal M_i.
\]
Now define the tagged disjoint-union input class
\[
\Omega_{\mathcal U}:=\bigl( \{0\}\times\Omega_0 \bigr) \cup \bigl( \{1\}\times\Omega_1 \bigr)
\]
and the tagged disjoint-union output space
\[
\mathcal M_{\mathcal U}:=\bigl( \{0\}\times \mathcal M_0 \bigr)\cup\bigl( \{1\}\times \mathcal M_1 \bigr),
\]
equipped with the metric
\[
d_{\mathcal U}\bigl( (i,x),(j,y) \bigr) :=
\begin{cases}
	\min\{1,d_i(x,y)\},&i=j,\\
	2,&i\neq j.
\end{cases}
\]
This is a metric: each tagged component is open and closed in \(\mathcal M_{\mathcal U}\). Further define the target map
\[
\Xi_{\mathcal U}:\Omega_{\mathcal U}\to \mathcal M_{\mathcal U}, \quad \Xi_{\mathcal U}(i,A):=(i,\Xi_i(A)).
\]
For each \(f\in\Lambda_0\), define
\[
\widetilde f_0:\Omega_{\mathcal U}\to\mathbb C, \quad \widetilde f_0(0,A):=f(A), \quad \widetilde f_0(1,B):=0.
\]
For each \(g\in\Lambda_1\), define
\[
\widetilde g_1:\Omega_U\to\mathbb C, \quad \widetilde g_1(1,B):=g(B), \quad \widetilde g_1(0,A):=0.
\]
Finally, define the tag query
\[
\tau:\Omega_{\mathcal U}\to\mathbb C, \quad \tau(i,A):=i.
\]
Set
\[
\Lambda_{\mathcal U} := \{\widetilde f_0:f\in\Lambda_0\}\cup \{\widetilde g_1:g\in\Lambda_1\}\cup \{\tau\}.
\]
We check that
\[
\mathcal U:=(\Xi_{\mathcal U},\Omega_{\mathcal U},(\mathcal M_{\mathcal U},d_{\mathcal U}),\Lambda_{\mathcal U})
\]
is an SCI computational problem.

Let \((i,A),(j,B)\in\Omega_{\mathcal U}\) and assume
\[
\Xi_{\mathcal U}(i,A)\neq \Xi_{\mathcal U}(j,B).
\]
If \(i\neq j\), then
\[
\tau(i,A)=i\neq j=\tau(j,B),
\]
so \(\tau\in\Lambda_{\mathcal U}\) separates the two inputs.

If \(i=j\), then
\[
\Xi_i(A)\neq \Xi_i(B).
\]
Since \(\mathcal P_i\) is an SCI computational problem, \cref{def:SCIcompProb} yields some \(h\in\Lambda_i\) such that
\[
h(A)\neq h(B).
\]
If \(i=0\), then
\[
\widetilde h_0(0,A)=h(A)\neq h(B)=\widetilde h_0(0,B).
\]
If \(i=1\), then
\[
\widetilde h_1(1,A)=h(A)\neq h(B)=\widetilde h_1(1,B).
\]
Thus \(\mathcal U\) satisfies the consistency condition from \cref{def:SCIcompProb}.

We now prove
\[
\mathcal P_i\leq^\mathscr R_{\mathrm{fq}} \mathcal U,
\]
where $i=0,1$. Fix \(i\in\{0,1\}\), and write \(j:=1-i\).

Define the encoding
\[
E_i:\Omega_i\to\Omega_{\mathcal U}, \quad E_i(A):=(i,A).
\]

Define the decoder
\[
D_i:\mathcal M_{\mathcal U}\to \mathcal M_i
\]
by
\[
D_i((i,x)):=x,\quad D_i((j,y)):=m_i^\ast.
\]
Since the tagged components of \(\mathcal M_{\mathcal U}\) are clopen, and \(D_i\) restricts to the identity on \(\{i\}\times \mathcal M_i\) and to a constant map on \(\{j\}\times \mathcal M_j\), the map \(D_i\) is continuous. Hence
\[
D_i\in \mathscr R_{\mathrm{cont}}\subseteq \mathscr R_{\mathrm{Bor}},
\]
so in either case $D_i\in R$.

For every \(A\in\Omega_i\),
\[
D_i(\Xi_{\mathcal U}(E_i(A))) = D_i(i,\Xi_i(A)) = \Xi_i(A).
\]
Thus the target relation holds and it remains to simulate the target queries from \(\Lambda_{\mathcal U}\).

First, let \(f\in\Lambda_i\). Then
\[
\widetilde f_i(E_i(A)) = \widetilde f_i(i,A) = f(A).
\]
So we may take
\[
m_{\widetilde f_i}:=1, \quad \gamma_{\widetilde f_i,1}:=f, \quad \vartheta_{\widetilde f_i}(z):=z.
\]

Second, let \(g\in\Lambda_j\). Then on the image of \(E_i\),
\[
\widetilde g_j(E_i(A)) = \widetilde g_j(i,A) = 0.
\]
Hence we may take
\[
m_{\widetilde g_j}:=1, \quad \gamma_{\widetilde g_j,1}:=q_i, \quad \vartheta_{\widetilde g_j}(z):=0.
\]

Third, for the tag query \(\tau\), one has on the image of \(E_i\),
\[
\tau(E_i(A))=\tau(i,A)=i.
\]
Hence we may take
\[
m_\tau:=1,\quad \gamma_{\tau,1}:=q_i, \quad \vartheta_\tau(z):=i.
\]

Therefore every query in \(\Lambda_{\mathcal U}\) is reconstructed from finitely many source queries of \(\mathcal P_i\). Thus
\[
\mathcal P_i\leq^\mathscr R_{\mathrm{fq}} \mathcal U.
\]
Since \(i\in\{0,1\}\) was arbitrary, both reductions hold.
\end{proof}

\begin{proposition}[Downward directedness on the nonempty-input subclass]\label{prop:downward-directed-nonempty}
Let
\[
R\in\{\mathscr R_{\mathrm{cont}}, \mathscr R_{\mathrm{Bor}}\}.
\]
Let
\[
\mathcal P_i=(\Xi_i,\Omega_i,(\mathcal M_i,d_i),\Lambda_i),
\]
where $i=0,1$, be SCI computational problems belonging to \(\mathcal C_{\neq\varnothing}\). Then there exists an SCI computational problem \(\mathcal L\) such that
\[
\mathcal L\leq^\mathscr R_{\mathrm{fq}} \mathcal P_0 \quad\text{and}\quad \mathcal L\leq^\mathscr R_{\mathrm{fq}} \mathcal P_1.
\]
In particular, the preorder \(\leq^\mathscr R_{\mathrm{fq}}\) is downward directed on \(\mathcal C_{\neq\varnothing}\).
\end{proposition}

\begin{proof}
Fix
\[
\mathcal P_i=(\Xi_i,\Omega_i,(\mathcal M_i,d_i),\Lambda_i),
\]
where $i=0,1$, with \(\Omega_i\neq\varnothing\). Choose
\[
A_i^\ast\in\Omega_i,
\]
where $i=0,1$ again. Define
\[
\Omega_{\mathcal L}:=\{\ast\}, \quad \mathcal M_{\mathcal L}:=\{0\},
\]
where \(\mathcal M_{\mathcal L}\) carries the unique metric \(d_{\mathcal L}\). Let
\[
\Xi_{\mathcal L}:\Omega_{\mathcal L} \to \mathcal M_{\mathcal L}, \quad \Xi_{\mathcal L}(\ast):=0.
\]
Define
\[
c:\Omega_{\mathcal L} \to\mathbb C, \quad c(\ast):=0,
\]
and set
\[
\Lambda_{\mathcal L}:=\{c\}.
\]
Then
\[
\mathcal L:=(\Xi_{\mathcal L},\Omega_{\mathcal L},(\mathcal M_{\mathcal L},d_{\mathcal L}),\Lambda_{\mathcal L})
\]
is an SCI computational problem.

We now prove
\[
\mathcal L\leq^\mathscr R_{\mathrm{fq}} \mathcal P_i
\]
for $i=0,1$. Fix \(i\in\{0,1\}\). Define the encoding
\[
E_i:\Omega_{\mathcal L}\to\Omega_i, \quad E_i(\ast):=A_i^\ast.
\]

Define the decoder
\[
D_i:\mathcal M_i\to \mathcal M_{\mathcal L}, \quad D_i(y):=0.
\]
This map is constant, hence continuous, so
\[
D_i\in \mathscr R_{\mathrm{cont}}\subseteq \mathscr R_{\mathrm{Bor}}.
\]
Therefore \(D_i\in R\). For the unique source input \(\ast\),
\[
D_i(\Xi_i(E_i(\ast))) = D_i(\Xi_i(A_i^\ast)) = 0 = \Xi_L(\ast).
\]
So the target relation holds.

Now fix any \(f\in\Lambda_i\). Since
\[
E_i(\ast)=A_i^\ast,
\]
the value
\[
f(E_i(\ast))=f(A_i^\ast)
\]
is constant. Hence we may take
\[
m_f:=1, \quad \gamma_{f,1}:=c, \quad \vartheta_f(z):=f(A_i^\ast).
\]
Then
\[
f(E_i(\ast)) = f(A_i^\ast) = \vartheta_f(c(\ast)) = \vartheta_f(\gamma_{f,1}(\ast)).
\]
Thus all clauses of \cref{def:decoder-regular-fq-transport} are satisfied, and therefore
\[
\mathcal L\leq^\mathscr R_{\mathrm{fq}} \mathcal P_i.
\]
Since \(i\in\{0,1\}\) was arbitrary, the claim follows.
\end{proof}

    \bibliographystyle{alpha}  
    \bibliography{bib.bib}
\end{document}